\documentclass[11 pt]{amsart}

\usepackage{amsfonts}
\usepackage{amssymb}
\usepackage{amsmath}
\usepackage{latexsym}
\usepackage{mathrsfs}
\usepackage{amsthm}
\usepackage{verbatim}
\usepackage{amscd}

\newtheorem{thrm}{Theorem}[section]
\newtheorem{lem}[thrm]{Lemma}
\newtheorem{cor}[thrm]{Corollary}
\newtheorem{prop}[thrm]{Proposition}
\newtheorem{conj}[thrm]{Conjecture}

\theoremstyle{definition}
\newtheorem{defn}[thrm]{Definition}

\newtheorem{exmple}[thrm]{Example}
\newtheorem{rmk}[thrm]{Remark}
\newtheorem{ques}[thrm]{Question}

\newtheorem{constr}[thrm]{Construction}
\newtheorem{conv}[thrm]{Convention}

\DeclareMathOperator{\Eff}{\overline{Eff}}

\DeclareMathOperator{\Nef}{Nef}

\DeclareMathOperator{\Chow}{Chow}

\DeclareMathOperator{\Supp}{Supp}

\DeclareMathOperator{\vol}{vol}

\DeclareMathOperator{\mc}{mc}

\DeclareMathOperator{\rmc}{rmc}

\DeclareMathOperator{\mob}{mob}

\DeclareMathOperator{\ratmob}{ratmob}

\DeclareMathOperator{\ch}{ch}

\DeclareMathOperator{\chdim}{chdim}

\DeclareMathOperator{\wmc}{wmc}

\DeclareMathOperator{\wmob}{wmob}

\DeclareMathOperator{\mult}{mult}

\begin{document}

\title{Volume-type functions for numerical cycle classes}
\author{Brian Lehmann}
\thanks{The author was supported by NSF Award 1004363.}
\address{Department of Mathematics, Boston College  \\
Chestnut Hill, MA \, \, 02467}

\email{lehmannb@bc.edu}

\begin{abstract}
A numerical equivalence class of $k$-cycles is said to be big if it lies in the interior of the closed cone generated by effective classes.  We construct analogues for arbitrary cycle classes of the volume function for divisors which distinguishes big classes from boundary classes.
\end{abstract}

\maketitle

\section{Introduction}

Let $X$ be an integral projective variety over an algebraically closed field.  We will let $N_{k}(X)_{\mathbb{Z}}$ denote the group of numerical classes of $k$-cycles on $X$ and $N_{k}(X) := N_{k}(X)_{\mathbb{Z}} \otimes \mathbb{R}$.  The pseudo-effective cone $\Eff_{k}(X) \subset N_{k}(X)$ is defined to be the closure of the cone generated by all effective $k$-cycles; this cone encodes the homology of all $k$-dimensional subvarieties of $X$. While the pseudo-effective cone has been thoroughly studied for divisors and curves, much less is known about cycles in general. 

One of the most basic problems concerning $\Eff_{k}(X)$ is to find geometric criteria that distinguish classes on the interior of the cone -- known as big classes -- from boundary classes.
This question has interesting links to a number of other geometric problems (see for example \cite[Theorem 0.8]{voisin10}, \cite[Remark 6.4]{djv13}, and Section \ref{ratmob0cycles}).   Our goal is to give several geometric characterizations of big cycles similar to well-known criteria for divisors.



An important tool for understanding big divisor classes is the volume function.  The volume of a Cartier divisor $L$ is the asymptotic rate of growth of dimensions of sections of $L$.  More precisely, if $X$ has dimension $n$,
\begin{equation*}
\vol(L) := \limsup_{m \to \infty} \frac{\dim H^{0}(X,\mathcal{O}_{X}(mL))}{m^{n}/n!}.
\end{equation*}
It turns out that the volume is an invariant of the numerical class of $L$ and satisfies many advantageous geometric properties.  On a smooth variety $X$, divisors with positive volume are precisely the divisors with big numerical class.

Our first generalization of the volume function is geometric in nature.  One can interpret the volume of a divisor $L$ as an asymptotic measurement of the number of  general points contained in members of $|mL|$ as $m$ increases.  \cite{delv11} suggests studying a similar notion for arbitrary cycles.  Given a class $\alpha \in N_{k}(X)_{\mathbb{Z}}$, we define
\begin{equation*}
\mc(\alpha) = \max \left\{ b \in \mathbb{Z}_{\geq 0} \, \left| \, \begin{array}{c} \textrm{any }b\textrm{ general points of } X \textrm{ are contained}  \\ \textrm{in an effective cycle of class } \alpha  \end{array} \right. \right\}.
\end{equation*}
We would then like to understand the asymptotic behavior of $\mc(m\alpha)$ as $m$ increases.  The ``expected'' growth behavior can be predicted by considering complete intersections in $\mathbb{P}^{n}$ as in Example \ref{ciexample}: for the hyperplane class $\alpha$ on 
 projective space we have $\mc(m\alpha) \sim Cm^{\frac{n}{n-k}}$ for some constant $C$.  

\begin{defn}
Let $X$ be an integral projective variety of dimension $n$ and suppose $\alpha \in N_{k}(X)_{\mathbb{Z}}$ for $0 \leq k < n$.  The mobility of $\alpha$ is
\begin{equation*}
\mob(\alpha) = \limsup_{m \to \infty} \frac{ \mc(m\alpha) }{m^{\frac{n}{n-k}}/n!}
\end{equation*}
\end{defn}

The mobility function shares many of the important properties of the volume function for divisors.  
Our first theorem shows that bigness is characterized by positive mobility, confirming \cite[Conjecture 6.5]{delv11}.

\begin{thrm} \label{maintheorem}
Let $X$ be an integral projective variety.  Then $\mob$ extends uniquely from $N_{k}(X)_{\mathbb{Z}}$ to a continuous homogeneous function on $N_{k}(X)$.  In particular, $\alpha \in N_{k}(X)$ is big if and only if $\mob(\alpha)>0$.
\end{thrm}



The following examples illustrate how the mobility captures basic geometric information about a class.  They also show that the mobility is difficult to compute; however, see Question \ref{mainconj} for a conjectural intersection-theoretic description.

\begin{exmple}
If $X$ is a smooth projective variety and $L$ is a Cartier divisor then $\mob([L]) = \vol(L)$ as shown in Example \ref{divisormob}.  However, the definition of mobility makes sense for a numerical Weil divisor class on any integral projective variety.  \cite{fl14} verifies that if $X$ admits a resolution of singularities $\phi: X' \to X$, then the mobility of a divisor class $\alpha$ is
\begin{equation*}
\mob(\alpha) = \sup_{\beta \in N_{k}(X'), \, \phi_{*}\beta = \alpha} \vol(\beta)
\end{equation*}
\end{exmple}

\begin{exmple}
Let $\ell$ denote the class of a line on $\mathbb{P}^{3}$.  The mobility of $\ell$ is determined by an enumerative question: what is the minimal degree of a curve in $\mathbb{P}^{3}$ going through $b$ general points?

It turns out that the answer to this question is not known (even asymptotically as the degree increases).  \cite{perrin87} conjectures that the ``optimal'' curves are complete intersections of two divisors of equal degree, which would imply that $\mob(\ell)=1$.  We return to this interesting question in Section \ref{p3mobility}.
\end{exmple}

\begin{exmple}
We define the rational mobility of a class $\alpha \in N_{k}(X)_{\mathbb{Z}}$ in a similar way by counting the number of general points lying on cycles in a fixed rational equivalence class inside of $\alpha$ (see Definition \ref{mobdefn}).  Rational mobility is interesting even for $0$-cycles.

Let $A_{0}(X)$ denote the set of rational equivalence classes of $0$-cycles on $X$.  Recall that $A_{0}(X)$ is said to be representable if the addition map $X^{(r)} \to A_{0}(X)_{\deg(r)}$ is surjective for some $r>0$.  In Section \ref{ratmob0cycles} we show for normal varieties over $\mathbb{C}$ that $A_{0}(X)$ is representable if and only if the rational mobility of the class of a point is the maximal possible value $(\dim X)!$.
\end{exmple}

Our next generalization of the volume function comes from intersection theory.  In addition to the volume function for divisors, \cite{xiao15} has recently defined an interesting intersection-theoretic volume function for curves.  The two definitions can be unified in the following way.  For $\alpha, \beta \in N_{k}(X)$, we write $\alpha \preceq \beta$ if $\beta - \alpha \in \Eff_{k}(X)$.   


\begin{defn}
Let $X$ be an integral projective variety of dimension $n$ and suppose $\alpha \in N_{k}(X)$ for $0 \leq k < n$.  Define
\begin{equation*}
\widehat{\vol}(\alpha) := \sup_{\phi, A} \{ A^{n} \}
\end{equation*}
as $\phi: Y \to X$ varies over all birational models of $X$ and as $A$ varies over all big and nef $\mathbb{R}$-Cartier divisors on $Y$ such that $\phi_{*}[A^{n-k}] \preceq \alpha$.
\end{defn}

When $\alpha$ is not in the interior of the pseudo-effective cone, the set of suitable divisors $A$ is empty and this expression should be interpreted as returning $0$.

\begin{exmple}
When $X$ is smooth and $L$ is a Cartier divisor, the theory of Fujita approximations (extended by \cite{takagi07} to arbitrary characteristic) shows that $\widehat{\vol}([L]) = \vol(L)$.
\end{exmple}

\begin{exmple}
Suppose that $B$ is a big and nef $\mathbb{R}$-Cartier divisor and set $\alpha = [B^{n-k}]$.  Then Example \ref{ktexample} shows that $\widehat{\vol}(\alpha) = \vol(B)$.  Many more examples for curve classes are computed in \cite{lx15}.
\end{exmple}

$\widehat{\vol}$ satisfies the most basic properties of a volume-type function.

\begin{thrm} \label{widehatvolthrm}
Let $X$ be an integral projective variety.  Then $\widehat{\vol}$ is a continuous homogeneous function on $N_{k}(X)$.  In particular, $\alpha \in N_{k}(X)$ is big if and only if $\widehat{\vol}(\alpha)>0$.
\end{thrm}


While this function is easier to compute than the mobility, it is unclear how it relates to the geometry of the  cycles of class proportional to $\alpha$.  This interpretation is provided by our main question.

\begin{ques} \label{mainconj}
Let $X$ be an integral projective variety of dimension $n$ and suppose $\alpha \in \Eff_{k}(X)$ for some $0 < k < n$.  Then is
\begin{equation*}
\widehat{\vol}(\alpha) = \mob(\alpha)?
\end{equation*}
\end{ques}

Theorem \ref{volmobineq} proves the inequality $\leq$.  Question \ref{mainconj} is known for divisor classes (even in the singular case), and \cite{lx15} shows that for curves it suffices to prove the special case when $\alpha = A^{n-1}$ for an ample divisor $A$.  For classes of intermediate codimension, one may need to broaden the definition of $\widehat{\vol}$ from ample divisors to other kinds of positive classes to obtain an equality.

\bigskip

Our final generalization of the volume function extrapolates between the two previous ones: in some examples it can be computed using intersection theory but it retains the flavor of the mobility.  The key idea is that singular points of  cycles should contribute more to the mobility count.  This convention better reflects the intersection theory on the blow-up of the points, as the strict transform of a cycle which is singular at a point will be have larger intersection against the exceptional divisor than the strict transform of a smooth cycle.  

Following a suggestion of R.~Lazarsfeld, we define the weighted mobility count of a class $\alpha \in N_{k}(X)_{\mathbb{Z}}$ as:
\begin{equation*}
\wmc(\alpha) =\max \left\{ b \in \mathbb{Z}_{\geq 0} \, \left| \, \begin{array}{c} \textrm{there is a } \mu \in \mathbb{Z}_{> 0} \textrm{ and an effective cycle of }  \\
\textrm{ class } \mu \alpha \textrm{ through any }b\textrm{ points of } X \textrm{ with}  \\ \textrm{multiplicity at least }\mu \textrm{ at each point}  \end{array} \right. \right\}.
\end{equation*}

This definition has the effect of counting singular points with a higher ``weight''.  It is designed to be compatible with the calculation of multipoint Seshadri constants -- see Section \ref{weightedmobilitysection} for details.  Just as with our other constants, the expected growth rate on a variety of dimension $n$ is $\wmc(m\alpha) \sim Cm^{n/n-k}$, suggesting the following definition.

\begin{defn}
Let $X$ be an integral projective variety of dimension $n$ and suppose $\alpha \in N_{k}(X)_{\mathbb{Z}}$ for $0 \leq k < n$.  The weighted mobility of $\alpha$ is
\begin{equation*}
\wmob(\alpha) = \limsup_{m \to \infty} \frac{ \wmc(m\alpha) }{m^{\frac{n}{n-k}}}
\end{equation*}
\end{defn}

The rescaling factor $n!$ is now omitted to ensure that the hyperplane class on $\mathbb{P}^{n}$ has weighted mobility $1$.

\begin{thrm}
Let $X$ be an integral projective variety.  Then $\wmob$ is a continuous homogeneous function on $N_{k}(X)$.  In particular, $\alpha \in N_{k}(X)$ is big if and only if $\wmob(\alpha)>0$.
\end{thrm}

\begin{exmple}
Suppose that $X$ is smooth over an uncountable algebraically closed field and that $L$ is a Cartier divisor on $X$.  Then Example \ref{wmobdivisor} shows that $\wmob([L]) = \vol(L)$.
\end{exmple}

\begin{exmple}
Suppose that $X$ is an integral projective variety over an uncountable algebraically closed field and that $\alpha = H^{n-k}$ where $H$ is a big and nef $\mathbb{R}$-divisor.  Then Example \ref{wmobci} shows that $\wmob(\alpha) = \vol(H)$.
\end{exmple}

\cite{lx15} shows that there is an equality $\wmob(\alpha) = \widehat{\vol}(\alpha)$ for curve classes.



\subsection{Acknowledgements}

My sincere thanks to A.M.~Fulger for numerous discussions and for his suggestions.  I am grateful to R.~Lazarsfeld for numerous conversations and for suggesting the approach in Section \ref{weightedmobilitysection}.  Thanks to C.~Voisin for recommending a number of improvements on an earlier draft.  I am grateful to B.~Bhatt and B.~Hassett  for helpful conversations and to Z.~Zhu and X.~Zhao for reading a draft of the paper.

\section{Preliminaries} \label{prelimsection}

Throughout we work over a fixed algebraically closed field $K$.  A variety will mean a quasiprojective scheme of finite type over $K$ (which may be reducible and non-reduced).  We will often use the following special case of \cite[Th\'eor\`eme 5.2.2]{gr71}.

\begin{thrm}[\cite{gr71}, Th\'eor\`eme 5.2.2]
Let $f: X \to S$ be a projective morphism of varieties such that some component of $X$ dominates $S$.  There is a birational morphism $\pi: S' \to S$ such that the morphism $f': X' \to S'$ is flat, where $X' \subset X \times_{S} S'$ is the closed subscheme defined by the ideal of sections whose support does not dominate $S'$.
\end{thrm}

\subsection{Cycles}
We will use the conventions of \cite{fulton84} concerning the theory of cycles and intersection products and refer to \cite{fl13} for the foundations of numerical groups.  Using the definition of numerical triviality of \cite[Chapter 19]{fulton84}, we let $N_{k}(X)_{\mathbb{Z}}$ denote the abelian group of numerical equivalence classes of $k$-cycles on $X$.  
We also define
\begin{align*}
N_{k}(X)_{\mathbb{Q}} & := N_{k}(X)_{\mathbb{Z}} \otimes \mathbb{Q} \\
N_{k}(X) & := N_{k}(X)_{\mathbb{Z}} \otimes \mathbb{R}
\end{align*}
When $Z$ is a Cartier divisor on $X$, \cite[Example 19.2.3]{fulton84} indicates how to construct pullback maps $f^{*}: N_{k}(X) \to N_{k-1}(Z)$.   
 For a cycle $Z$ on $X$, we let $[Z]$ denote the numerical class of $Z$, which can be naturally thought of as an element in $N_{k}(X)_{\mathbb{Z}}$, $N_{k}(X)_{\mathbb{Q}}$, or $N_{k}(X)$.  If $\alpha$ is the class of an effective cycle $Z$, we say that $\alpha$ is an effective class.

We denote the dual vector space of $N_{k}(X)$ by $N^{k}(X)$.  The intersections of Chern classes with rational equivalence classes descend to give maps $N^{r}(X) \times N_{k}(X) \to N_{k-r}(X)$. 

\begin{conv} \label{dimconv} When we discuss $k$-cycles on an integral projective variety $X$, we will always implicitly assume that $0 \leq k < \dim X$.  This allows us to focus on the interesting range of behaviors without repeating hypotheses.
\end{conv}

\begin{defn}
Let $X$ be a projective variety.  The pseudo-effective cone $\Eff_{k}(X) \subset N_{k}(X)$ is the closure of the cone generated by all classes of effective $k$-cycles.   $\Eff_{k}(X)$ is a full-dimensional salient cone by \cite[Theorem 0.2]{fl13}.  The big cone is the interior of the pseudo-effective cone.  The cone in $N^{k}(X)$ dual to the pseudo-effective cone is known as the nef cone and denoted $\Nef^{k}(X)$.

We say that $\alpha \in N_{k}(X)$ is pseudo-effective (resp.~big) if it lies in the pseudo-effective cone (resp.~big cone), and $\beta \in N^{k}(X)$ is nef if it lies in the nef cone.    For $\alpha,\alpha' \in N_{k}(X)$ we write $\alpha \preceq \alpha'$ when $\alpha' - \alpha$ is pseudo-effective.
\end{defn}



The following lemma records a basic property of pseudo-effective cycles.

\begin{lem}[\cite{fl13}, Corollary 3.20] \label{surjpushforward}
Let $f: X \to Y$ be a surjective morphism of integral projective varieties.  Then $f_{*}\Eff_{k}(X) = \Eff_{k}(Y)$.\end{lem}



\subsection{Analytic lemmas}

We are interested in invariants constructed as asymptotic limits of functions on $N_{k}(X)_{\mathbb{Z}}$.  The following lemmas will allow us to conclude several important properties of these functions directly from some easily verified conditions.

\begin{lem}[\cite{lazarsfeld04} Lemma 2.2.38] \label{lazlemma}
Let $f: \mathbb{N} \to \mathbb{R}_{\geq 0}$ be a function.  Suppose that for any $r,s \in \mathbb{N}$ with $f(r) > 0$ we have $f(r+s) \geq f(s)$.
Then for any $k \in \mathbb{R}_{>0}$ the function $g: \mathbb{N} \to \mathbb{R} \cup \{ \infty \}$ defined by
\begin{equation*}
g(r) := \limsup_{m \to \infty} \frac{f(mr)}{m^{k}}
\end{equation*}
satisfies $g(cr) = c^{k}g(r)$ for any $c,r \in \mathbb{N}$.
\end{lem}

\begin{rmk}
Although \cite[Lemma 2.2.38]{lazarsfeld04} only explicitly address the volume function, the essential content of the proof is the more general statement above.  In particular $k$ does not need to be an integer.
\end{rmk}

\begin{lem} \label{easyconelem}
Let $V$ be a finite dimensional $\mathbb{Q}$-vector space and let $C \subset V$ be a salient full-dimensional closed convex cone.  Suppose that $f: V \to \mathbb{R}_{\geq 0}$ is a function satisfying
\begin{enumerate}
\item $f(e) > 0$ for any $e \in C^{int}$,
\item there is some constant $c > 0$ so that $f(me) = m^{c}f(e)$ for any $m \in \mathbb{Q}_{>0}$ and $e \in C$, and
\item for every $v \in C^{int}$ and $e \in C^{int}$ we have $f(v+e) \geq f(v)$.
\end{enumerate}
Then $f$ is locally uniformly continuous on $C^{int}$.
\end{lem}

\section{Families of cycles} \label{cyclefamilysection}

Although there are several different notions of a family of cycles in the literature, the theory we will develop is somewhat insensitive to the precise choices.  It will be most convenient to use a simple geometric definition.

\begin{defn}  \label{familydef}
Let $X$ be a projective variety.  A family of $k$-cycles on $X$ consists of an integral variety $W$, a reduced closed subscheme $U \subset W \times X$, and an integer $a_{i}$ for each component $U_{i}$ of $U$, such that for each component $U_{i}$ of $U$ the first projection map $p: U_{i} \to W$ is flat dominant of relative dimension $k$.  If each $a_{i} \geq 0$ we say that we have a family of effective cycles.  We say that $\sum a_{i}U_{i}$ is the cycle underlying the family.  

In this situation $p: U \to W$ will denote the first projection map and $s: U \to X$ will denote the second projection map unless otherwise specified.  We will usually denote a family of $k$-cycles using the notation $p: U \to W$, with the rest of the data implicit.

For a closed point $w \in W$, the base change $w \times_{W} U_{i}$ is a subscheme of $X$ of pure dimension $k$ and thus defines a fundamental $k$-cycle $Z_{i}$ on $X$.  The cycle-theoretic fiber of $p: U \to W$ over $w$ is defined to be the cycle $\sum a_{i}Z_{i}$ on $X$.  We will also call these cycles the members of the family $p$.
\end{defn}

\begin{defn}
Let $X$ be a projective variety.  We say that a family of $k$-cycles $p: U \to W$ on $X$ is a rational family if every cycle-theoretic fiber lies in the same rational equivalence class.
\end{defn}


For any projective variety $X$, \cite{kollar96} constructs a Chow variety $\Chow(X)$.  Any family of cycles in the sense of Definition \ref{familydef} is also a family of cycles in the refined sense of \cite{kollar96}; this is an immediate consequence of \cite[I.3.14 Lemma]{kollar96} and \cite[I.3.15 Corollary]{kollar96}.  Thus, if $p: U \to W$ is a family of cycles, there is an induced map $\ch: W \dashrightarrow \Chow(X)$ defined on the open locus where $W$ is normal.  For a more in-depth discussion of the relationship between Definition \ref{familydef} and \cite{kollar96} we refer to \cite{lehmann14}.

The following constructions show how to construct families of cycles from subsets $U \subset W \times X$.

\begin{constr}[Cycle version] \label{cycletofamilyconstr}
Let $X$ be a projective variety and let $W$ be an integral variety.  Suppose that $Z = \sum a_{i}V_{i}$ is a $(k+\dim W)$-cycle on $W \times X$ such that the first projection maps each $V_{i}$ dominantly onto $W$.  Let $W^{0} \subset W$ be the (non-empty) open locus over which every projection $p: V_{i} \to W$ is flat and let $U \subset \Supp(Z)$ denote the preimage of $W^{0}$.  Then the map $p: U \to W^{0}$ defines a family of cycles where we assign the coefficient $a_{i}$ to the component $V_{i} \cap U$ of $U$.
\end{constr}

\begin{constr}[Subscheme version] \label{equidimsubschemeconstr}
Suppose that $Y$ is a reduced variety and that $X$ is a projective variety.  Let $\tilde{U} \subset Y \times X$ be a closed subscheme such that the fibers of the projection $p: \tilde{U} \to Y$ are equidimensional of dimension $k$.  There is a natural way to construct a finite collection of families of effective cycles associated to the subscheme $\tilde{U}$.

Consider the image $p(\tilde{U})$ (with its reduced induced structure).  Let $\{ \tilde{W}_{j} \}$ denote the irreducible components of $p(\tilde{U})$.  For each there is a non-empty open subset $W_{j} \subset \tilde{W}_{j}$ such that the restriction of $p$ to each component of $p^{-1}(W_{j})_{red}$ is flat.  Since furthermore $p$ has equidimensional fibers, we obtain a family of effective $k$-cycles $p_{j}: U_{j} \to W_{j}$ where $U_{j} = p^{-1}(W_{j})_{red}$ and we assign coefficients so that the cycle underlying the family $p_{j}$ coincides with the fundamental cycle of $p^{-1}(W_{j})$.  We can then replace $\tilde{U}$ by the closed subscheme obtained by taking the base change to $p(\tilde{U}) - \cup_{j} W_{j}$ and repeat.  The end result is a collection of families $p_{i}: U_{i} \to W_{i}$ parametrizing the cycles contained in $\tilde{U}$.

If $p(\tilde{U})$ is irreducible and we are interested only in the generic behavior of the cycles in $\tilde{U}$, we can stop after the first step to obtain a single family of cycles.
\end{constr}

It will often be helpful to replace a family $p: U \to W$ by a slightly modified version using a flattening argument.

\begin{lem} \label{goodfamilymodification}
Let $X$ be a projective variety and let $p: U \to W$ be a family of effective cycles on $X$.  Then there is a normal projective variety $W'$ that is birational to $W$ and a family of cycles $p': U' \to W'$ such that $p$ and $p'$ agree over an open subset of the base.
\end{lem}



\subsection{Geometry of families}

\begin{defn}
Let $X$ be a projective variety and let $p: U \to W$ be a family of effective cycles on $X$.
We say that $p$ is an irreducible family if $U$ only has one component.
For any component $U_{i}$ of $U$, we have an associated irreducible family $p_{i}: U_{i} \to W$ (with coefficient $a_{i}$).
\end{defn}

Given a family of cycles $p: U \to W$, consider the cycle $V$ underlying $U$ on $W \times X$.  We can then perform various cycle-theoretic constructions to $V$ and reconstruct a family of cycles using Construction \ref{cycletofamilyconstr}.  Note that the resulting family will usually only be defined over an open subset of the base $W$.  In this way we can define:
\begin{itemize}
\item Proper pushforward families.
\item Flat pullback families (which increase the dimension of the members of the family by the relative dimension of the map).
\item Restrictions of families to subvarieties of $W$ via base change of the flat map $p$.
\item Family sums: given two families $p: U \to W$ and $q: S \to T$, we construct a family over an open subset of $W \times T$ whose cycle-theoretic fibers are sums of the fibers of $p$ and $q$.
\item Strict transform families: given a rational map $\phi: X \dashrightarrow Y$, we first remove any components of $U$ contained in the locus where $\phi$ is not an isomorphism, and take the strict transform of the rest.
\item Intersections against the members of a linear series $|L|$: this defines a family of $k-1$ cycles over an open subset of $W \times |L|$.
\end{itemize}

\section{Mobility count} \label{mcsection}

The mobility count of a family of effective cycles can be thought of informally as a count of how many general points of $X$ are contained in members of the family.  Although we are mainly interested in families of cycles, it will be helpful to set up a more general framework.

\begin{defn} \label{mcdefn}
Let $X$ be an integral projective variety and let $W$ be a reduced variety.  Suppose that $U\subset W \times X$ is a subscheme and let $p: U \to W$ and $s: U \to X$ denote the projection maps.  The mobility count $\mc(p)$ of the morphism $p$ is the maximum non-negative integer $b$ such that the map
\begin{equation*}
U \times_{W} U \times_{W} \ldots \times_{W} U \xrightarrow{s \times s \times \ldots \times s} X \times X \times \ldots \times X
\end{equation*}
is dominant, where we have $b$ terms in the product on each side.  (If the map is dominant for every positive integer $b$, we set $\mc(p) = \infty$.)

For $\alpha \in N_{k}(X)_{\mathbb{Z}}$, the mobility count of $\alpha$, denoted $\mc(\alpha)$, is defined to be the largest mobility count of any family of effective cycles representing $\alpha$.  We define the rational mobility count $\rmc(\alpha)$ in the analogous way by restricting our attention to rational families.
\end{defn}

\begin{exmple} \label{mcandvar}
Let $X$ be an integral projective variety and let $p: U \to W$ be a family of effective $k$-cycles on $X$.  Then $\mc(p) \leq (\dim W)/(\dim X - k)$.  Indeed, if the map of Definition \ref{mcdefn} is dominant then dimension considerations show that $\mc(p)k + \dim W \geq \mc(p) \dim X$.
\end{exmple}


\begin{exmple} \label{mcdiv}
Let $X$ be an normal projective variety and let $L$ be a Cartier divisor on $X$.  Let $p: U \to W$ denote the family of effective divisors on $X$ defined by the complete linear series for $L$.  Then
\begin{equation*}
\mc(p) = \dim H^{0}(X,\mathcal{O}_{X}(L)) - 1.
\end{equation*}
Indeed, it is easy to see that the set of divisors in our family which contain a general point of $X$ corresponds to a codimension $1$ linear subspace of $|L|$.  Furthermore, an easy induction argument shows that the collection of $b$ sufficiently general points corresponds to a collection of $b$ hyperplanes which intersect transversally.  Thus, the maximum number of general points contained in a member of $|L|$ is exactly $\dim \mathbb{P}(|L|)$, and using the incidence correspondence one identifies this number as $\mc(p)$ as well.
\end{exmple}

\begin{lem} \label{fibercontainedmc}
Let $X$ be an integral projective variety.  Let $W$ be a reduced variety and let $p: U \to W$ denote a closed subscheme of $W \times X$.  Suppose that $T$ is another reduced variety and $q: S \to T$ is a closed subscheme of $T \times X$ such that every fiber of $p$ over a closed point of $W$ is contained in a fiber of $q$ over some closed point of $T$ (as subsets of $X$).  Then $\mc(p) \leq \mc(q)$.
\end{lem}

\begin{proof}
The conditions imply that for any $b>0$, the $s^{b}$-image of any fiber of $p^{b}: U^{\times_{W} b} \to W$ over a closed point of $W$ is set-theoretically contained in the image of a fiber of $q^{b}: S^{\times_{T} b} \to T$ over a closed point of $T$ (as subsets of $X^{\times b}$).  The statement follows.
\end{proof}

\begin{prop} \label{opensubsetsforcycles}
Let $X$ be an integral projective variety.
\begin{enumerate}
\item Let $W$ be an integral variety and let $U \subset W \times X$ be a closed subscheme such that $p: U \to W$ is flat.  For an open subvariety $W^{0} \subset W$ let $p^{0}: U^{0} \to W^{0}$ be the base change to $W^{0}$.  Then $\mc(p) = \mc(p^{0})$.
\item Let $p: U \to W$ be a family of effective cycles on $X$.  For an open subvariety $W^{0} \subset W$ let $p^{0}: U^{0} \to W^{0}$ be the restriction family.  Then $\mc(p) = \mc(p^{0})$.
\item Let $W$ be a normal integral variety and let $U \subset W \times X$ be a closed subscheme such that:
\begin{itemize}
\item Every fiber of the first projection map $p: U \to W$ has the same dimension.
\item Every component of $U$ dominates $W$ under $p$.
\end{itemize}
Let $W^{0} \subset W$ be an open subset and $p^{0}: U^{0} \to W^{0}$ be the preimage of $W^{0}$.  Then $\mc(p) = \mc(p^{0})$.
\end{enumerate}
\end{prop}

\begin{rmk}
Proposition \ref{opensubsetsforcycles} indicates that the mobility count is insensitive to the choice of definition of a family of effective cycles.  It also shows that the mobility count only depends on general members of the family.
\end{rmk}

\begin{proof}
(1) The map $p^{b}: U^{\times_{W} b} \to W$ is proper flat, so that every component of $U^{\times_{W} b}$ dominates $W$.  Then $(U^{0})^{\times_{W^{0}} b}$ is dense in $U^{\times_{W} b}$ for any $b$.  Thus $\mc(p^{0}) = \mc(p)$.

(2) Let $\{ U_{i} \}$ denote the irreducible components of $U$.  Every irreducible component of $U^{\times_{W} b}$ is contained in a product of the $U_{i}$ over $W$.  Since each $p|_{U_{i}}: U_{i} \to W$ is flat, we can apply the same argument as in (1).

(3) The inequality $\mc(p) \geq \mc(p^{0})$ is clear.  To show the converse inequality, we may suppose that $U$ is reduced.  We may also shrink $W^{0}$ and assume that $p^{0}$ is flat.

Let $p': U' \to W'$ be a flattening of $p$ via the birational morphism $\phi: W' \to W$.  We may ensure that $\phi$ is an isomorphism over $W^{0}$.  Choose a closed point $w \in W$ and let $T \subset W'$ be the set-theoretic preimage.  Since $W$ is normal $T$ is connected.

Choose a closed point $w' \in T$.  By construction the fiber $U'_{w'}$ is set theoretically contained in $U_{w}$ (as subsets of $X$).  Since they have the same dimension, $U'_{w'}$ is a union of components of $U_{w}$.  Since $p'$ is flat over $T$ and $T$ is connected, in fact $U'_{w'}$ and $U_{w}$ have the same number of components and thus are set-theoretically equal.  Applying Lemma \ref{fibercontainedmc} and part (1) we see that $\mc(p) \leq \mc(p') = \mc(p^{0})$.
\end{proof}

We can now describe how the mobility count changes under certain geometric constructions of families of cycles.

\begin{lem} \label{ignorecomponents}
Let $X$ be an integral projective variety and let $p: U \to W$ be a family of effective $k$-cycles.  Suppose that $U$ has a component $U_{i}$ whose image in $X$ is contained in a proper subvariety.  Then $\mc(p) = \mc(p')$ where $p'$ is the family defined by removing $U_{i}$ from $U$.
\end{lem}

\begin{proof} This is immediate from the definition.
\end{proof}

\begin{lem} \label{stricttransformmc}
Let $\psi: X \dashrightarrow Y$ be a birational morphism of integral projective varieties.  Let $p: U \to W$ be a family of effective $k$-cycles on $X$ and let $p'$ denote the strict transform family on $Y$.  Then $\mc(p) = \mc(p')$.
\end{lem}

\begin{proof}
By Lemma \ref{ignorecomponents} we may assume that every component of $U$ dominates $X$.  Using Proposition \ref{opensubsetsforcycles} (2), we may replace $p$ by the restricted family $p^{0}: U^{0} \to W^{0}$, where $W^{0}$ is the locus of definition of the strict transform family $p': U' \to W^{0}$.  The statement is then clear using the fact that the morphisms $(U^{0})^{\times_{W^{0}}b} \to X^{\times b}$ and $(U')^{\times_{W^{0}}b} \to Y^{\times b}$ are birationally equivalent for every $b$.
\end{proof}

\begin{lem} \label{familysummc}
Let $X$ be an integral projective variety.  Suppose that $W$ and $T$ are reduced varieties and that $p_{1}: U \to W$ and $p_{2}: S \to T$ are closed subschemes of $W \times X$ and $T \times X$ respectively.  Let $q: V \to W \times T$ denote the subscheme
\begin{equation*}
U \times T \cup W \times S \subset W \times T \times X.
\end{equation*}
Then $\mc(q) = \mc(p_{1}) + \mc(p_{2})$.

In particular, if $p_{1}$ and $p_{2}$ are families of effective $k$-cycles, then the mobility count of the family sum is the sum of the mobility counts.
\end{lem}

\begin{proof}
Set $b_{1} = \mc(p_{1})$ and $b_{2} = \mc(p_{2})$.  There is a dominant projection map
\begin{equation*}
\left( U^{\times_{W} b_{1}} \times T \right) \times_{W \times T} \left( W \times S^{\times_{T} b_{2}} \right) \to X^{\times (b_{1} + b_{2})}.
\end{equation*}
Since the domain is naturally a subscheme of $V^{\times_{W \times T} b_{1} + b_{2}}$, we obtain $\mc(q) \geq \mc(p_{1}) + \mc(p_{2})$.

Conversely, any irreducible component of $V^{\times_{W \times T} c}$ is (up to reordering the terms) a subscheme of
\begin{equation*}
\left( U^{\times_{W} c_{1}} \times T \right) \times_{W \times T} \left( W \times S^{\times_{T} c_{2}} \right)
\end{equation*}
for some non-negative integers $c_{1}$ and $c_{2}$ with $c = c_{1} + c_{2}$ where the map to $X^{\times b}$ is component-wise.  This yields the reverse inequality.

To extend the lemma to the family sum, first replace $p_{1}$ and $p_{2}$ by their restrictions to the normal locus of $W$ and $T$ respectively; this does not change the mobility count by Proposition \ref{opensubsetsforcycles} (2).  Then by Proposition \ref{opensubsetsforcycles} (3) the mobility count of the family sum is the same as the mobility count of the subscheme $U \times T \cup W \times S$.
\end{proof}

\begin{cor} \label{mobilitydecompositioncor}
Let $X$ be an integral projective variety and let $p: U \to W$ be a family of effective $k$-cycles.  Let $p_{i}: U_{i} \to W$ denote the irreducible components of $U$.  Then $\mc(p) \leq \sum_{i} \mc(p_{i})$.
\end{cor}

\begin{proof}
Let $q: S \to T$ denote the family sum of the $p_{i}$.  By Lemma \ref{fibercontainedmc} we have $\mc(p) \leq \mc(q)$; by Lemma \ref{familysummc} $\mc(q) = \sum_{i} \mc(p_{i})$.
\end{proof}

\subsection{Families of divisors}

We next analyze families of effective divisors.  The goal is to find bounds on the mobility count that depend only on the numerical class of the divisor.  The key result is Corollary \ref{basicvarestimate}, which is the base case of inductive arguments used in the following sections.

\begin{prop} \label{divisorrestriction}
Let $X$ be an integral projective variety of dimension $n$.  Let $\alpha \in N_{n-1}(X)_{\mathbb{Z}}$ and suppose that $A$ is a very ample divisor such that $\alpha - [A]$ is not pseudo-effective.  Then for a general element $H \in |A|$
\begin{equation*}
\mc(\alpha) \leq \mc_{H}(\alpha \cdot H).
\end{equation*}
\end{prop}

\begin{proof}
For a general $H \in |A|$ we have that $H$ is integral.  Let $p: U \to W$ be a family of effective $(n-1)$-cycles representing $\alpha$.  By Lemma \ref{goodfamilymodification} we may suppose that $s: U \to X$ is projective and $W$ is normal.  Set $b = \mc(p)$ so that
\begin{equation*}
U^{ \times_{W} b} \to X^{\times b}
\end{equation*}
is surjective.  Since surjectivity is preserved by base change, we see that the base change of $s$ to $H$ is still surjective, yielding a closed subset $U_{H} \subset W \times H$ with mobility count $\geq b$.  Note that since $\alpha - [A]$ is not pseudo-effective, no divisor in the family $p$ can contain $H$ in its support, so that $U_{H}$ has pure relative dimension $n-2$ over $W$.

Since no divisor in the family $p$ contains $H$ in its support, we have a well-defined intersection family $p \cdot H$ on $H$.  Over an open subset $W^{0}$ of the parameter space, this family coincides set-theoretically with the closed subset $U_{H}$.  By Proposition \ref{opensubsetsforcycles}.(3), the mobility count of $U_{H}$ does not change upon restriction to the open subset of the base, so that $\mc_{H}(p \cdot H) \geq b$.  The result follows by varying $p$ over all families representing $\alpha$.
\end{proof}

Applying this proposition inductively, we easily obtain:

\begin{cor} \label{basicvarestimate}
Let $X$ be an integral projective variety of dimension $n$ and let $\alpha \in N_{n-1}(X)_{\mathbb{Z}}$.  
\begin{enumerate}
\item Suppose that $A$ is a very ample Cartier divisor on $X$ and $s$ is a positive integer such that $\alpha \cdot A^{n-1} < sA^{n}$.  Then
\begin{equation*}
\mc(\alpha) < s^{n} A^{n}.
\end{equation*}
\item Suppose $n \geq 2$.  Let $A$ and $H$ be very ample divisors and let $s$ be a positive integer such that $\alpha - [H]$ is not pseudo-effective and $\alpha \cdot A^{n-2} \cdot H < s A^{n-1} \cdot H$.  Then
\begin{equation*}
\mc(\alpha) < s^{n-1} A^{n-1} \cdot H.
\end{equation*}
\end{enumerate}
\end{cor}



An easy induction using long exact sequences yields a similar statement for sections of line bundles.

\begin{lem} \label{basicsectionestimate}
Let $X$ be an equidimensional projective variety of dimension $n$ and let $A$ be a very ample Cartier divisor on $X$.  Then
\begin{equation*}
h^{0}(X,\mathcal{O}_{X}(A)) \leq (n+1)A^{n}.
\end{equation*}
\end{lem}


\section{The mobility function} \label{mobilitysection}

As suggested by \cite{delv11}, we will define the mobility of a class $\alpha \in N_{k}(X)_{\mathbb{Z}}$ to be the asymptotic growth rate of the number of general points contained in cycles representing multiples of $\alpha$.  We prove that big classes are precisely those with positive mobility, confirming \cite[Conjecture 6.5]{delv11}.

Recall that by Convention \ref{dimconv} we only consider $k$-cycles for $0 \leq k < \dim X$.  The first step is:

\begin{prop} \label{mcbound}
Let $X$ be an integral projective variety of dimension $n$ and let $\alpha \in N_{k}(X)_{\mathbb{Z}}$.  Fix a very ample divisor $A$ and choose a positive constant $c<1$ so that $h^{0}(X,mA) \geq \lfloor c m^{n} \rfloor$ for every positive integer $m$.  Then any family $p: U \to W$ representing $\alpha$ has
\begin{equation*}
\mc(p) \leq (n+1)2^{n} \left(\frac{2(k+1)}{c} \right)^{\frac{n}{n-k}} (\alpha \cdot A^{k})^{\frac{n}{n-k}}A^{n}.
\end{equation*}
In particular, there is some constant $C$ so that $\mc(m\alpha) \leq Cm^{\frac{n}{n-k}}$.
\end{prop}

We will develop a bound that does not depend on the constant $c$ in Theorem \ref{mobprecisebound}.

\begin{proof}
By Lemma \ref{basicsectionestimate}, the support $Z$ of any effective cycle representing $\alpha$ satisfies
\begin{equation*}
h^{0}(X,\mathcal{I}_{Z}(dA)) \geq \lfloor c d^{n} \rfloor - (k+1)d^{k}(\alpha \cdot A^{k})
\end{equation*}
for any positive integer $d$.  Thus an effective cycle representing $\alpha$ is set-theoretically contained in an element of $|\lceil d \rceil A|$ as soon as $d$ is sufficiently large to make the right hand side greater than $1$, and in particular, for
\begin{equation*}
d = \left( \frac{2(k+1)}{c} \right)^{\frac{1}{n-k}} (\alpha \cdot A^{k})^{\frac{1}{n-k}}.
\end{equation*}
Let $q: \tilde{U} \to \mathbb{P}(| \lceil d \rceil A|)$ denote the family of divisors defined by the linear series.  By Lemma \ref{fibercontainedmc} we have $\mc(p) \leq \mc(q)$.  Since $c<1$, $d \geq 1$ so that $\lceil d \rceil < 2d$.  Applying Lemma \ref{basicsectionestimate} again, Example \ref{mcdiv} indicates that
\begin{equation*}
\mc(p) \leq h^{0}(X,\lceil d \rceil A) -1 < (n+1)2^{n} \left(\frac{2(k+1)}{c} \right)^{\frac{n}{n-k}} (\alpha \cdot A^{k})^{\frac{n}{n-k}}A^{n}.
\end{equation*}
\end{proof}

Furthermore, one easily verifies a growth rate $\mc(m\alpha) \geq \lfloor Cm^{\frac{n}{n-k}}\rfloor$ is always achieved by a complete intersection of ample divisors.  Proposition \ref{mcbound} thus indicates that we should make the following definition.

\begin{defn} \label{mobdefn}
Let $X$ be an integral projective variety and let $\alpha \in N_{k}(X)_{\mathbb{Z}}$.  The mobility of $\alpha$ is
\begin{equation*}
\mob_{X}(\alpha) = \limsup_{m \to \infty} \frac{\mc(m\alpha)}{m^{n/n-k}/n!}.
\end{equation*}
We will omit the subscript $X$ when the ambient variety is clear from the context.  We define the rational mobility $\ratmob(\alpha)$ in an analogous way using $\rmc$.
\end{defn}

The coefficient $n!$ is justified by Section \ref{p3mobility}.  We verify in Example \ref{divisormob} that the mobility agrees with the volume function for Cartier divisors on a smooth integral projective variety.



\subsection{Basic properties} We now turn to the basic properties of the mobility function.

\begin{lem} \label{rescalingmob}
Let $X$ be an integral projective variety and let $\alpha \in N_{k}(X)_{\mathbb{Z}}$.  Fix a positive integer $a$.  Then $\mob(a\alpha) = a^{\frac{n}{n-k}}\mob(\alpha)$ (and similarly for $\ratmob$).
\end{lem}

\begin{proof}
If $\mc(r\alpha) > 0$ then $r\alpha$ is represented by a family of effective cycles.  Thus $\mc((r+s)\alpha) \geq \mc(s\alpha)$ for any positive integer $s$ by the additivity of mobility count under family sums as in Lemma \ref{familysummc}.  Conclude by Lemma \ref{lazlemma}.
\end{proof}

Lemma \ref{rescalingmob} allows us to extend the definition of mobility to any $\mathbb{Q}$-class by homogeneity, obtaining a function $\mob: N_{k}(X)_{\mathbb{Q}} \to \mathbb{R}_{\geq 0}$.

\begin{exmple} \label{divisormob}
Let $X$ be a smooth projective integral variety of dimension $n$.  For any Cartier divisor $L$ on $X$ we have
\begin{equation*}
\ratmob([L]) = \mob([L]) = \vol(L).
\end{equation*}
To prove this, note first that by Example \ref{mcandvar} we have
\begin{equation*}
\rmc(m[L]) \leq \mc(m[L]) \leq \dim \Chow(X,m[L])
\end{equation*}
where $\Chow(X,m[L])$ is the locus of $\Chow(X)$ parametrizing divisors of class $m[L]$.  While this rightmost term may be greater than $h^{0}(X,\mathcal{O}_{X}(mL))-1$, the difference is bounded by a polynomial of degree $n-1$ in $m$ as in \cite[Proposition 2.2.43]{lazarsfeld04}.  By taking asymptotics, we find that $\ratmob([L]) \leq \mob([L]) \leq \vol(L)$.  In particular, if $L$ is not big then we must have equalities everywhere.

Conversely, by Example \ref{mcdiv} we have $\rmc(m[L]) \geq h^{0}(X,\mathcal{O}_{X}(mL))-1$, so that taking asymptotics $\ratmob([L]) \geq \vol(L)$.
\end{exmple}

\begin{lem} \label{mobadditive}
Let $X$ be an integral projective variety.  Suppose that $\alpha, \beta \in N_{k}(X)_{\mathbb{Q}}$ are classes such that some positive multiple of each is represented by an effective cycle.  Then $\mob(\alpha + \beta) \geq \mob(\alpha) + \mob(\beta)$ (and similarly for $\ratmob$).
\end{lem}

\begin{proof}
We may verify the inequality after rescaling $\alpha$ and $\beta$ by the same positive constant $c$.  Thus we may suppose that every multiple of each is represented by an effective cycle.  Using the additivity of mobility counts under family sums as in Lemma \ref{familysummc}, we see that
\begin{equation*}
\mc(m(\alpha + \beta)) \geq \mc(m\alpha) + \mc(m\beta)
\end{equation*}
and the conclusion follows.
\end{proof}

\begin{cor} \label{bigposcor}
Let $X$ be an integral projective variety and let $\alpha \in N_{k}(X)_{\mathbb{Q}}$ be a big class.  Then $\mob(\alpha) > 0$ (and similarly for $\ratmob$).
\end{cor}

\begin{thrm} \label{mobcontbigcone}
Let $X$ be an integral projective variety.  The function $\mob: N_{k}(X)_{\mathbb{Q}} \to \mathbb{R}_{\geq 0}$ is locally uniformly continuous on the interior of $\Eff_{k}(X)_{\mathbb{Q}}$ (and similarly for $\ratmob$).
\end{thrm}

Theorem \ref{mobcontinuous} extends this result to prove that $\mob$ is continuous on all of $N_{k}(X)$. 

\begin{proof}
The conditions (1)-(3) in Lemma \ref{easyconelem} are verified by Corollary \ref{bigposcor}, Lemma \ref{rescalingmob}, and Lemma \ref{mobadditive}.
\end{proof}

The mobility should also have good concavity properties.  Here is a strong conjecture in this direction:

\begin{conj} \label{mobconvexconj}
Let $X$ be an integral projective variety.  Then $\mob$ is a log-concave function on $\Eff_{k}(X)$: for any classes $\alpha, \beta \in \Eff_{k}(X)$ we have
\begin{equation*}
\mob(\alpha + \beta)^{\frac{n-k}{n}} \geq \mob(\alpha)^{\frac{n-k}{n}} + \mob(\beta)^{\frac{n-k}{n}}
\end{equation*}
\end{conj}

We note one other basic property:

\begin{prop} \label{mobpushforwardprop}
Let $\pi: X \to Y$ be a dominant generically finite morphism of integral projective varieties.  For any $\alpha \in N_{k}(X)_{\mathbb{Q}}$ we have $\mob(\pi_{*}\alpha) \geq \mob(\alpha)$.
\end{prop}

\begin{proof}
For sufficiently divisible $m$ set $p_{m}$ to be a family of effective $k$-cycles representing $m \alpha$ of maximal mobility count.  The pushforward family $\pi_{*}p_{m}$ clearly has the same mobility count as $p_{m}$ and represents $\pi_{*}(m\alpha)$, giving the result.
\end{proof}

\subsection{Mobility and bigness}

We now show that big cycles are precisely those with positive mobility:

\begin{thrm} \label{mobilityandbigness}
Let $X$ be an integral projective variety and let $\alpha \in N_{k}(X)_{\mathbb{Q}}$.  The following statements are equivalent:
\begin{enumerate}
\item $\alpha$ is big.
\item $\ratmob(\alpha) > 0$.
\item $\mob(\alpha) > 0$.
\end{enumerate}
\end{thrm}

The implication (1) $\implies$ (2) follows from Corollary \ref{bigposcor} and (2) $\implies$ (3) is obvious.  The implication (3) $\implies$ (1) is a consequence of the more precise statement in Corollary \ref{iitakacor}.


\begin{exmple} \label{mobilityofcurves}
Let $\alpha$ be a curve class on an integral projective complex variety $X$.  \cite[Theorem 2.4]{8authors} shows that if two general points of $X$ can be connected by an effective curve whose class is proportional to $\alpha$, then $\alpha$ is big.  Theorem \ref{mobilityandbigness} is a somewhat weaker statement in this situation.
\end{exmple}

For positive integers $n$ and for $0 \leq k < n$, define $\epsilon_{n,k}$ inductively by setting $\epsilon_{n,n-1} = 1$ and
\begin{equation*}
\epsilon_{n,k} = \frac{ \frac{n-k-1}{n-k}\epsilon_{n-1,k} }{\frac{n-1}{n-k-1} - \epsilon_{n-1,k}}.
\end{equation*}
For positive integers $n$ and for $0 \leq k < n$, define $\tau_{n,k}$ inductively by setting $\tau_{n,n-1} = 1$ and
\begin{equation*}
\tau_{n,k} = \min \left\{ \frac{n-k-1}{n-1} \tau_{n-1,k}, \frac{ \frac{n-k-1}{n-k}\tau_{n-1,k} }{\frac{n-1}{n-k-1} - \tau_{n-1,k}} \right\}.
\end{equation*}
It is easy to verify that $0 < \tau_{n,k} \leq \epsilon_{n,k} \leq \frac{1}{n-k}$ and that the last inequality is strict as soon as $n-k>1$.

\begin{thrm} \label{mobprecisebound}
Let $X$ be an integral projective variety and let $\alpha \in N_{k}(X)_{\mathbb{Z}}$.  Let $A$ be a very ample divisor and let $s$ be a positive integer such that $\alpha \cdot A^{k} < sA^{n}$.  Then
\begin{enumerate}
\item
\begin{equation*}
\mc(\alpha) < 2^{kn+3n} s^{\frac{n}{n-k}}A^{n}.
\end{equation*}
\item Suppose furthermore that $\alpha - [A]^{n-k}$ is not pseudo-effective.  Then
\begin{equation*}
\mc(\alpha) < 2^{kn+3n} s^{\frac{n}{n-k} - \epsilon_{n,k}} A^{n}.
\end{equation*}
\item Suppose that $t$ is a positive integer such that $t \leq s$ and $\alpha - t[A]^{n-k}$ is not pseudo-effective.  Then
\begin{equation*}
\mc(\alpha) < 2^{kn+3n}  s^{\frac{n}{n-k} - \tau_{n,k}} t^{\tau_{n,k}} A^{n}.
\end{equation*}
\end{enumerate}
\end{thrm}


\begin{proof}
We prove (1) by induction on the dimension $n$ of $X$.  The induction step may reduce the codimension $n-k$ of our cycle class by at most $1$.  Thus, for the base case it suffices to consider when $k=0$ or when $n$ is arbitrary and $n-k = 1$.  The first is elementary and the second is proved by Corollary \ref{basicvarestimate} (1).

Let $p: U \to W$ be a family of effective $k$-cycles representing $\alpha$.  By Proposition \ref{opensubsetsforcycles} we may modify $p$ by Lemma \ref{goodfamilymodification} to assume $W$ is projective without changing the mobility count of $p$.  Choose a general divisor $H$ in the very ample linear series $|\lceil s^{\frac{1}{n-k}} \rceil A|$ on $X$ such that $H$ is integral and does not contain the image of any component of $U$.  We can associate several families of subschemes to $p$ and to $H$.
\begin{itemize}
\item Consider the base change $U \times_{X} H$.  We can view this as a subscheme $U_{H} \subset W \times H$ with projection map $\pi: U_{H} \to W$.
\item We can intersect the family $p$ with the divisor $H$ to obtain a family of effective $(k-1)$ cycles $q: S \to T$ on $H$.  Note that $T$ is an open subset of $W$; we may shrink $T$ so that it is normal.  By Lemma \ref{goodfamilymodification} we may extend $q$ over a projective closure of $T$ which we continue to denote by $q: S \to T$.
\item Let $V \subset U_{H}$ be the reduced closed subset consisting of points whose local fiber dimension for $\pi$ attains the maximal possible value $k$.  The map $\pi|_{V}: V \to W$ has equidimensional fibers of dimension $k$.  Thus we can associate to $\pi|_{V}$ a collection of families $p_{i}: V_{i} \to W_{i}$ as in Construction \ref{equidimsubschemeconstr}.  We will think of these as families of effective $k$-cycles on $H$.
\end{itemize}
It will be useful to combine $q$ and the $p_{i}$ as follows.  Let $\widetilde{W}$ denote the disjoint union of the irreducible varieties $T \times W_{i}$ as we vary over all $i$.  This yields the subscheme $S \times \left( \sqcup_{i} W_{i} \right) \cup T \times \left( \sqcup_{i} V_{i} \right)$ of $\widetilde{W} \times X$.  We denote the first projection map by $\tilde{p}$.

Using the universal property, one can see that
\begin{equation*}
(U \times_{X} H)^{\times_{W} b} \cong (U^{\times_{W} b}) \times_{X^{\times b}} H^{\times b}.
\end{equation*}
Since the base change of a surjective map is surjective, we see that $\mc_{H}(\pi) \geq \mc_{X}(p)$.  Furthermore, by Krull's principal ideal theorem every component of a fiber of $\pi$ over a closed point of $W$ has dimension $k$ or $k-1$.  In particular, any member of the family $\pi$ is set theoretically contained in a member of the family $\tilde{p}$.  Applying in order the inequality from the start of this paragraph, Lemma \ref{fibercontainedmc}, and Lemma \ref{familysummc}, we obtain
\begin{equation*}
\mc_{X}(p) \leq \mc_{H}(\pi) \leq \mc_{H}(\tilde{p}) = \mc_{H}(q) + \sup_{i} \mc_{H}(p_{i}).
\end{equation*}
We will use induction to bound the two terms on the right, giving us our overall bound for $\mc_{X}(p)$.

The family $q$ of effective $(k-1)$-cycles on $H$ has class $\alpha \cdot H$.  Note that $(\alpha \cdot H) \cdot A|_{H}^{k-1} < sA|_{H}^{n-1}$.  By induction on the dimension of the ambient variety,
\begin{align*}
\mc_{H}(q) & < 2^{(k-1)(n-1) + 3(n-1)}  s^{\frac{n-1}{n-k}}(A|_{H}^{n-1}) \\
& \leq 2^{(k-1)(n-1) + 3(n-1)} s^{\frac{n-1}{n-k}} (2s^{\frac{1}{n-k}} A^{n}) \\
& \leq 2^{(k-1)(n-1) + 3(n-1)+1} s^{\frac{n}{n-k}} A^{n}.
\end{align*}

Next consider a family $p_{i}$ of effective $k$-cycles on $H$.  Let $\alpha_{i}$ denote the corresponding class.  Let $j: H \to X$ be the inclusion; by construction, it is clear that $\alpha - j_{*}\alpha_{i}$ is the class of an effective cycle.  In particular
\begin{equation*}
\alpha_{i} \cdot A|_{H}^{k} \leq \alpha \cdot A^{k} < \lceil s^{\frac{n-k-1}{n-k}} \rceil A|_{H}^{n-1}.
\end{equation*}
By induction on the dimension of the ambient variety,
\begin{align*}
\mc_{H}(p_{i}) & < 2^{(n-1)k + 3(n-1)}  \lceil s^{\frac{n-k-1}{n-k}} \rceil^{\frac{n-1}{n-k-1}}(A|_{H}^{n-1}) \\
& \leq 2^{(n-1)k + 3(n-1)} 2^{\frac{n-1}{n-k-1}}  s^{\frac{n-1}{n-k}}(2s^{\frac{1}{n-k}}A^{n}) \\
& \leq 2^{(n-1)k + 3(n-1) + k + 2} s^{\frac{n}{n-k}} A^{n}.
\end{align*}
By adding these contributions, we see that
\begin{equation*}
\mc_{X}(p) \leq 2^{kn+3n}  s^{\frac{n}{n-k}}A^{n}.
\end{equation*}

\bigskip

(2) is proved in a similar way.  The argument is by induction on the codimension $n-k$ of $\alpha$.  The base case -- when $n$ is arbitrary and $n-k=1$ -- is a consequence of Corollary \ref{basicvarestimate} (2) (applied with $H=A$).

Let $p: U \to W$ be a family of effective $k$-cycles representing $\alpha$.  Set $c := \frac{1}{n-k} - \epsilon_{n,k}$.  Let $H$ be an integral element of $|\lceil s^{c} \rceil A|$ that does not contain the image of any component of $U$.  We construct the families $q: S \to T$ and $p_{i}: V_{i} \to W_{i}$ just as in (1).  The same argument shows that
\begin{equation*}
\mc_{X}(p) \leq \mc_{H}(q) + \sup_{i} \mc_{H}(p_{i}).
\end{equation*}

The family $q$ of effective $(k-1)$-cycles on $H$ has class $\alpha \cdot H$.  Note that $(\alpha \cdot H) \cdot A|_{H}^{k-1} < sA|_{H}^{n-1}$.  By (1), we have
\begin{align*}
\mc_{H}(q) & < 2^{(k-1)(n-1)+3(n-1)} s^{\frac{n-1}{n-k}}(A|_{H}^{n-1}) \\
& \leq 2^{(k-1)(n-1)+3(n-1)} s^{\frac{n-1}{n-k}} (2s^{c} A^{n}) \\
& \leq 2^{(k-1)(n-1)+3(n-1) + 1}  s^{\frac{n}{n-k} - \epsilon_{n,k}} A^{n}.
\end{align*}

Next consider the family $p_{i}$ of effective $k$-cycles on $H$.  Let $\alpha_{i}$ denote the class of the family $p_{i}$ on $H$.  Let $j: H \to X$ be the inclusion; by construction, it is clear that $\alpha - j_{*}\alpha_{i}$ is the class of an effective cycle.  In particular
\begin{align*}
\alpha_{i} \cdot A|_{H}^{k} \leq \alpha \cdot A^{k}  < \left\lceil s^{1-c} \right\rceil A|_{H}^{n-1}.
\end{align*}
Note furthermore that $\alpha_{i} - [A|_{H}]^{n-1-k}$ is not pseudo-effective; otherwise it would push forward to a pseudo-effective class on $X$, contradicting the fact that $\alpha - [A]^{n-k}$ is not pseudo-effective.  By induction on the codimension of the cycle,
\begin{align*}
\mc_{H}(p_{i}) &  < 2^{k(n-1)+3(n-1)} \lceil s^{1-c} \rceil^{\frac{n-1}{n-k-1} - \epsilon_{n-1,k}}(A|_{H}^{n-1}) \\
 & \leq 2^{k(n-1)+3(n-1)} 2^{\frac{n-1}{n-k-1}}  s^{\frac{(1-c)(n-1)}{n-k-1} - (1-c)\epsilon_{n-1,k}}(2s^{c}A^{n}) \\
& \leq 2^{k(n-1)+3(n-1) + k + 2} s^{\frac{n}{n-k} - \epsilon_{n,k}} A^{n}.
\end{align*}
Adding the two contributions proves the statement as before.

\bigskip

The proof of (3) is also very similar.  The argument is by induction on the codimension $n-k$ of $\alpha$.  The base case -- when $n$ is arbitrary and $n-k=1$ -- is a consequence of Corollary \ref{basicvarestimate} (2) (applied with $H=tA$).

Let $p: U \to W$ be a family of effective $k$-cycles representing $\alpha$.  Set $c := \frac{1}{n-k} - \tau_{n,k}$.  Let $H$ be an integral element of $|\lceil s^{c} t^{\tau_{n,k}} \rceil A|$ that does not contain the image of any component of $U$.  We construct the families $q: S \to T$ and $p_{i}: V_{i} \to W_{i}$ just as in (1).  The same argument shows that
\begin{equation*}
\mc_{X}(p) \leq \mc_{H}(q) + \sup_{i} \mc_{H}(p_{i}).
\end{equation*}

The family $q$ of effective $(k-1)$-cycles on $H$ has class $\alpha \cdot H$.  Note that $(\alpha \cdot H) \cdot A|_{H}^{k-1} < sA|_{H}^{n-1}$.  By (1), we have
\begin{align*}
\mc_{H}(q) & < 2^{(k-1)(n-1)+3(n-1)} s^{\frac{n-1}{n-k}}(A|_{H}^{n-1}) \\
& \leq 2^{(k-1)(n-1)+3(n-1)} s^{\frac{n-1}{n-k}} (2s^{c}t^{\tau_{n,k}} A^{n}) \\
& \leq 2^{(k-1)(n-1)+3(n-1) + 1}  s^{\frac{n}{n-k} - \tau_{n,k}} t^{\tau_{n,k}} A^{n}.
\end{align*}

Next consider the family $p_{i}$ of effective $k$-cycles on $H$.  Let $\alpha_{i}$ denote the class of the family $p_{i}$ on $H$.  Let $j: H \to X$ be the inclusion; by construction, it is clear that $\alpha - j_{*}\alpha_{i}$ is the class of an effective cycle.  In particular
\begin{align*}
\alpha_{i} \cdot A|_{H}^{k} \leq \alpha \cdot A^{k} & < \left\lceil s^{1-c} t^{-\tau_{n,k}} \right\rceil A|_{H}^{n-1}.
\end{align*}
Also, we have that
\begin{align*}
\alpha_{i} - \lceil t^{1-\tau_{n,k}}s^{-c} \rceil [A|_{H}]^{n-k-1}
\end{align*}
is not pseudo-effective, since the difference between $\alpha - t[A]^{n-k}$ and the push forward of this class to $X$ is pseudo-effective.  Finally, note that $\lceil s^{1-c}t^{-\tau_{n,k}} \rceil \geq \lceil t^{1-\tau_{n,k}}s^{-c} \rceil$ so that we may apply (3) inductively to the family $p_{i}$ with the constants $s' = \lceil s^{1-c}t^{-\tau_{n,k}} \rceil$ and $t' = \lceil t^{1-\tau_{n,k}}s^{-c} \rceil$.

There are two cases to consider.  First suppose that $t^{1-\tau_{n,k}}s^{-c} \geq 1$.  Then by induction on the codimension of the cycle,
\begin{align*}
\mc_{H}(p_{i}) &  < 2^{k(n-1)+3(n-1)}  \lceil s^{1-c} t^{-\tau_{n,k}} \rceil^{\frac{n-1}{n-k-1} - \tau_{n-1,k}} \\
& \qquad \qquad \lceil t^{1-\tau_{n,k}} s^{-c} \rceil^{\tau_{n-1,k}} (A|_{H}^{n-1}) \\
 & \leq 2^{k(n-1)+3(n-1)}  2^{\frac{n-1}{n-k-1}}  s^{\frac{(1-c)(n-1)}{n-k-1} - \tau_{n-1,k}} \\
 & \qquad \qquad t^{\tau_{n-1,k} - \frac{n-1}{n-k-1} \tau_{n,k}}(2s^{c}t^{\tau_{n,k}}A^{n}) \\
& \leq 2^{k(n-1)+3(n-1) + k + 2} s^{\frac{n}{n-k} - \tau_{n,k} + \left(\frac{n-1}{n-k-1}\tau_{n,k} - \tau_{n-1,k} \right)} \\
& \qquad \qquad t^{\tau_{n,k} + \left(\tau_{n-1,k} - \frac{n-1}{n-k-1}\tau_{n,k}\right)} A^{n}.
\end{align*}
Since $\tau_{n-1,k} \geq \frac{n-1}{n-k-1}\tau_{n,k}$, the part of the exponents in parentheses is non-positive for $s$ and non-negative for $t$.  By assumption $s \geq t$, so 
\begin{align*}
\mc_{H}(p_{i}) < 2^{kn+3(n-1) + 2} s^{\frac{n}{n-k} - \tau_{n,k}} t^{\tau_{n,k}} A^{n}
\end{align*}
Next suppose that $t^{1-\tau_{n,k}}s^{-c} < 1$.  Then by (2) we find
\begin{align*}
\mc_{H}(p_{i}) & < 2^{k(n-1)+3(n-1)} \lceil s^{1-c}t^{-\tau_{n,k}} \rceil^{\frac{n-1}{n-k-1} - \epsilon_{n-1,k}} (A|_{H}^{n-1}) \\
& \leq 2^{k(n-1)+3(n-1)}  \lceil s^{1-c} \rceil^{\frac{n-1}{n-k-1} - \tau_{n-1,k}} (A|_{H}^{n-1}) \\
& \leq 2^{k(n-1)+3(n-1)}  2^{\frac{n-1}{n-k-1}} s^{\frac{(1-c)(n-1)}{n-k-1} - (1-c)\tau_{n-1,k}} (2s^{c}t^{\tau_{n,k}}A^{n}) \\
& \leq 2^{kn + 3(n-1)+2} s^{\frac{n}{n-k} - \tau_{n,k}} t^{\tau_{n,k}} A^{n}.
\end{align*}
This upper bound for the two cases is the same; by adding it to the upper bound for $\mc_{H}(q)$ we obtain the desired upper bound for $\mc(p)$.
\end{proof}

We can apply Theorem \ref{mobprecisebound} (2) to any class $\alpha$ in $\partial \Eff_{k}(X) \cap N_{k}(X)_{\mathbb{Z}}$ to obtain the following corollary.

\begin{cor} \label{iitakacor}
Let $X$ be an integral projective variety and suppose that $\alpha \in N_{k}(X)_{\mathbb{Z}}$ is not big.  Let $A$ be a very ample divisor and let $s$ be a positive integer such that $\alpha \cdot A^{k} < sA^{n}$.  Then
\begin{equation*}
\mc(\alpha) < 2^{kn + 3n}(k+1)s^{\frac{n}{n-k} - \epsilon_{n,k}} A^{n}.
\end{equation*}
\end{cor}

\begin{rmk}
The exponent $\frac{n}{n-k} - \epsilon_{n,k}$ in Corollary \ref{iitakacor} is not optimal in general.  For example, \cite[Theorem 2.4]{8authors} shows that for a curve class $\alpha$ that is not big there is a positive constant $C$ such that $\mc(m\alpha) < Cm$.
\end{rmk}


\subsection{Continuity of mobility}

Theorem \ref{mobprecisebound} also allows us to prove the continuity of the mobility function.

\begin{thrm} \label{mobcontinuous}
Let $X$ be an integral projective variety.  Then the mobility function $\mob: N_{k}(X)_{\mathbb{Q}} \to \mathbb{R}$ can be extended to a continuous function on $N_{k}(X)$.
\end{thrm}

\begin{proof}
Note that $\mob$ can be extended to a continuous function on the interior of $\Eff_{k}(X)$ by Theorem \ref{mobcontbigcone}.  Furthermore $\mob$ is identically $0$ on every element in $N_{k}(X)_{\mathbb{Q}}$ not contained in $\Eff_{k}(X)$.  Thus it suffices to show that $\mob$ approaches $0$ for classes approaching the boundary of $\Eff_{k}(X)$.

Let $\alpha$ be a point on the boundary of $\Eff_{k}(X)$.  Fix $\mu > 0$; we show that there exists a neighborhood $U$ of $\alpha$ such that $\mob(\beta) < \mu$ for any class $\beta \in U \cap N_{k}(X)_{\mathbb{Q}}$.

Fix a very ample divisor $A$ and a positive integer $s$ such that $\alpha \cdot A^{k} < \frac{s}{2}A^{n}$.  Choose $\delta$ sufficiently small so that
\begin{equation*}
n! 2^{kn+3n+1}  s^{\frac{n}{n-k}} A^{n} \delta^{\tau_{n,k}} < \mu.
\end{equation*}
Let $U$ be a sufficiently small neighborhood of $\alpha$ so that:
\begin{itemize}
\item $\beta \cdot A^{k} < sA^{n}$ for every $\beta \in U$, and
\item $\beta - \delta s [A]^{n-k}$ is not pseudo-effective for every $\beta \in U$.
\end{itemize}
Suppose now that $\beta \in U \cap N_{k}(X)_{\mathbb{Q}}$ and that $m$ is any positive integer such that $m\beta \in N_{k}(X)_{\mathbb{Z}}$.  Then:
\begin{itemize}
\item $m\beta \cdot A^{k} < smA^{n}$ and
\item $m\beta - \lceil \delta ms \rceil [A]^{n-k}$ is not pseudo-effective.
\end{itemize}
Theorem \ref{mobprecisebound} shows that
\begin{equation*}
\mc(m\beta) < 2^{kn+3n} (ms)^{\frac{n}{n-k} - \tau_{n,k}} (\lceil \delta ms \rceil)^{\tau_{n,k}} A^{n}.
\end{equation*}
When $m$ is sufficiently large, $\lceil \delta ms \rceil \leq 2 \delta ms$, so we obtain for such $m$
\begin{equation*}
\mc(m\beta) < 2^{kn+3n+1} m^{\frac{n}{n-k}} s^{\frac{n}{n-k}} A^{n}  \delta^{\tau_{n,k}} < \frac{\mu}{n!} m^{\frac{n}{n-k}}
\end{equation*}
showing that $\mob(\beta)<\mu$ as desired.
\end{proof}


\section{Examples of mobility} \label{mobilityexamplesection}

The mobility seems difficult to calculate explicitly.  By analogy with the volume, one wonders whether the mobility is related to intersection numbers for ``sufficiently positive'' classes (just as the volume of an ample divisor is a self-intersection product).  In particular, we ask:

\begin{ques} \label{intersectiontheoreticques}
Let $X$ be an integral projective variety and let $H$ be an ample Cartier divisor.  For $0<k<n$, is
\begin{equation*}
\mob(H^{n-k}) = \vol(H)?
\end{equation*}
\end{ques}


This question is vastly generalized by Question \ref{mainconj}, but already this particular case is very interesting.


\begin{exmple} \label{ciexample}
It is not hard to show that $\mob(H^{n-k}) \geq H^{n}$ for an ample divisor $H$.  Indeed, by homogeneity we may suppose that $H$ is very ample and that all the higher cohomology of multiples of $H$ vanishes.  We can find a basepoint free family of divisors of class $m[H]$ containing $h^{0}(X,mH) - 1$ general points of $X$.  By taking complete intersections, we obtain a family of class $m^{n-k}[H^{n-k}]$ going through $h^{0}(X,mH)-1-n+k$ general points of $X$.  Taking a limit shows the inequality.
\end{exmple}

In the remainder of this section we discuss two examples in detail.  We will work over the base field $\mathbb{C}$ to cohere with the cited references.

\subsection{Curves on $\mathbb{P}^{3}$}  \label{p3mobility}

Let $\ell$ denote the class of a line on $\mathbb{P}^{3}$ over $\mathbb{C}$.  The mobility of $\ell$ is determined by the following enumerative question: what is the minimal degree of a curve in $\mathbb{P}^{3}$ going through $b$ very general points?  The answer is unknown (even in the asymptotic sense).

\cite{perrin87} conjectures that the ``optimal'' curves are the complete intersections of two hypersurfaces of degree $d$.  Indeed, among all curves not contained in a hypersurface of degree $(d-1)$, \cite{gp77} shows that these complete intersections have the largest possible arithmetic genus, and thus conjecturally the corresponding Hilbert scheme has the largest possible dimension.  For a family $p$ of smooth curves, \cite{perrin87} uses the Gruson-Peskine bounds to prove:
\begin{equation*}
\mc(p) \leq \frac{1}{2}d^{3/2} + O(d).
\end{equation*}

One can give a lower bound for the mobility count by explicitly constructing families of cycles.  Complete intersections of two hypersurfaces of degree $d$ have degree $d^{2}$ and pass through $\approx \frac{1}{6}d^{3}$ general points.  Letting $d$ go to infinity, we find the lower bound
\begin{equation*}
1 \leq \mob(\ell)
\end{equation*}
and conjecturally equality holds.

\begin{thrm} \label{moblinep3}
Let $\ell$ be the class of a line on $\mathbb{P}^{3}$.  Then
\begin{equation*}
1 \leq \mob(\ell) < 3.54.
\end{equation*}
\end{thrm}

The proof is to simply repeat the argument of Theorem \ref{mobprecisebound} with more careful constructions of families and better estimates.

\subsection{Rational mobility of points} \label{ratmob0cycles}

In this section we relate rational mobility with the theory of rational equivalence of $0$-cycles.  In order to cohere with the cited references, we work only with normal integral varieties $X$ over $\mathbb{C}$ (although the results easily extend to a more general setting). 
We let $A_{0}(X)$ denotes the group of rational-equivalence classes of $0$-cycles on $X$.  We will denote the $r$th symmetric power of $X$ by $X^{(r)}$; by \cite[I.3.22 Exercise]{kollar96} this is the component of $\Chow(X)$ parametrizing $0$-cycles of degree $r$.

\begin{rmk}
The universal family of $0$-cycles of degree $r$ (in the sense of Definition \ref{familydef}) is not $u: X^{\times r} \to X^{(r)}$ but a flattening of this map.  However, note that the rational mobility computations are the same whether we work with $u$ or a flattening by Lemma \ref{opensubsetsforcycles}.  For simplicity we will work with $u$ and $X^{(r)}$ despite the slight incongruity with Definition \ref{familydef}.
\end{rmk}

We start by recalling the results of \cite{roitman72} concerning $A_{0}(X)$.  
Consider the map $\gamma_{m,n}: X^{(m)} \times X^{(n)} \to A_{0}(X)$ sending $(p,q) \mapsto p-q$.  \cite[Lemma 1]{roitman72} shows that the fibers of $\gamma_{m,n}$ are countable unions of closed subvarieties.

A subset $V \subset A_{0}(X)$ is said to be irreducible closed if it is the $\gamma_{m,n}$-image of an irreducible closed subset $Y$ of $X^{(m)} \times X^{(n)}$ for some $m$ and $n$.  The dimension of such a subset $V$ is defined to be the dimension of $Y$ minus the minimal dimension of a component of a fiber of $\gamma_{m,n}|_{Y}$.  \cite[Lemma 9]{roitman72} shows that the dimension is independent of the choice of $Y$, $m$, and $n$.

\begin{lem} \label{containmentdimineq}
Let $V,W \subset A_{0}(X)$ be irreducible closed subsets with $V \subsetneq W$.  Then $\dim(V) < \dim(W)$.
\end{lem}

\begin{proof}
Let $Z \subset X^{(m)} \times X^{(n)}$ be an irreducible closed subset whose $\gamma_{m,n}$-image is $W$.  Let $Y \subset Z$ denote the preimage of $V$; \cite[Lemma 5]{roitman72} shows that $Y$ is a countable union of closed subsets.  By \cite[Lemma 6]{roitman72} some component $Y' \subset Y$ dominates $V$.  But then $\dim(Y') < \dim(Z)$, proving the statement.
\end{proof}

We are mainly interested in when $A_{0}(X)$ is an irreducible closed set.  This is equivalent to the following notion:

\begin{defn}
$A_{0}(X)$ is representable if there is a positive integer $r$ such that the addition map $a_{r}: X^{(r)} \to A_{0}(X)_{\deg r}$ is surjective.
\end{defn}

An alternative definition is that $\gamma_{m,n}$ is surjective for some choices of $m,n$; the two notions are equivalent (at least for smooth varieties) by \cite[Theorem 10.12]{voisin10}.
We now relate these notions to the rational mobility of $0$-cycles on $X$.

\begin{prop} \label{representabilityandratmob}
Let $X$ be a normal integral projective variety over $\mathbb{C}$ and let $\alpha$ denote the class of a point in $N_{0}(X)$.  Then the following are equivalent:
\begin{enumerate}
\item $A_{0}(X)$ is representable.
\item $\ratmob(\alpha) = n!$.
\item $\ratmob(\alpha) > n!/2$.
\end{enumerate}
\end{prop}

\begin{proof}
(1) $\implies$ (2).  Suppose that $A_{0}(X)$ is representable.  There is some positive integer $r$ such that the addition map $a_{r}: X^{(r)} \to A_{0}(X)_{\deg r}$ is surjective.  Fix $m>0$ and choose some class $\tau \in A_{0}(X)_{\deg(m+r)}$.  For any effective $0$-cycle $Z$ of degree $m$, there is an effective $0$-cycle $T_{Z}$ of degree $r$ such that $T_{Z} + Z \in \tau$.   As $Z \in X^{(m)}$ varies, the effective cycles $Z + T_{Z}$ are rationally equivalent, showing that $\rmc((m+r)\alpha) \geq m$ and $\ratmob(\alpha) = n!$.

(3) $\implies$ (1).  Suppose that $A_{0}(X)$ is not representable.  For every $m$, non-representability implies that $a_{m}(X^{(m)}) +  p \subsetneq a_{m+1}(X^{(m+1)})$ for some closed point $p$: if we had equality for every point, we would also have equality for every $m' > m$ and for every point, and one easily deduces the surjectivity of $a_{m}$.  Thus by Lemma \ref{containmentdimineq} $\dim(a_{m}(X^{(m)}))$ strictly increases in $m$.

Suppose that $\rmc(m\alpha) = b$.  This implies that there is some rational equivalence class $\tau$ of degree $m$ so that for any $p \in X^{(b)}$, there is an element $q \in X^{(m-b)}$ such that $p+q \in \tau$.  In particular, the subset $\tau - X^{(b)} \subset A_{0}(X)_{\deg m-b}$ is contained in $a_{m-b}(X^{(m-b)})$.  By Lemma \ref{containmentdimineq}, $\dim(a_{b}(X^{(b)})) \leq \dim(a_{m-b}(X^{(m-b)}))$.  But since these dimensions are strictly increasing in $m$ we must have $m \geq 2b$.  Thus we see that $\ratmob(\alpha) \leq n!/2$, proving the statement.
\end{proof}

\begin{exmple} \label{0cyclepositiveratmob}
Let $X$ be an integral projective variety of dimension $n$ and let $\alpha$ be the class of a point in $N_{0}(X)$.  Let $A$ be a very ample divisor on $X$; for sufficiently large $m$ we have $h^{0}(X,\mathcal{O}_{X}(mA)) \approx \frac{1}{n!}A^{n}$.  By taking complete intersections of $n$ elements of $|mA|$, we see that $\ratmob(\alpha) \geq 1$.
\end{exmple}

\begin{exmple}
Let $X$ be a smooth surface over $\mathbb{C}$ and let $\alpha$ be the class of a point.  By combining Example \ref{0cyclepositiveratmob} with Proposition \ref{representabilityandratmob} we see that there are two possibilites:
\begin{itemize}
\item $A_{0}(X)$ is representable and $\ratmob(\alpha) = 2$.
\item $A_{0}(X)$ is not representable and $\ratmob(\alpha) = 1$.
\end{itemize}
\end{exmple}

\section{Intersection-theoretic volume function} \label{inttheorysection}

In this section we study the function $\widehat{\vol}$ defined in the introduction.
In contrast to the mobility, one can readily compute this function in simple examples.


\begin{exmple} \label{ktexample}
Suppose that $B$ is a big and nef $\mathbb{R}$-Cartier divisor and that $\alpha = [B^{n-k}]$.  We claim that  $\widehat{\vol}(\alpha) = \vol(B)$.

Indeed, suppose that $\phi: Y \to X$ is a birational map and $A$ is a big and nef $\mathbb{R}$-Cartier divisor such that $\phi_{*}[A^{n-k}] \preceq \alpha$.  Recall that by the Khovanskii-Teissier inequalities (see for example \cite[Corollary 1.6.3 and Remark 1.6.5]{lazarsfeld04})
\begin{equation*}
A^{n-k} \cdot \phi^{*}B^{k} \geq \vol(A)^{n-k/n} \vol(B)^{k/n}.
\end{equation*}
Then we have
\begin{align*}
\vol(A) \leq & \left( \frac{A^{n-k} \cdot \phi^{*}B^{k}}{\vol(B)^{k/n}} \right)^{n/n-k} \\
\leq & \left( \frac{ \alpha \cdot B^{k}}{\vol(B)^{k/n}} \right)^{n/n-k} \\
= & \vol(B).
\end{align*}
Since clearly $\widehat{\vol}(\alpha) \geq \vol(B)$, we obtain the claimed equality.
\end{exmple}

\begin{exmple}
Suppose that $X$ is smooth and that $\alpha \in \Eff_{1}(X)$ is a curve class.  Then \cite{lx15} shows that $\widehat{\vol}(\alpha)$ agrees with the expression in \cite{xiao15} and so can be computed once one knows the nef cone of $X$:
\begin{equation*}
\widehat{\vol}(\alpha) = \inf_{A \textrm{ big and nef }\mathbb{R}\textrm{-divisor}} \left( \frac{A \cdot \alpha}{\vol(A)^{1/n}} \right)^{n/n-1}.
\end{equation*}
In fact, for curve classes the supremum in the definition is actually achieved by a divisor on $X$ -- there is no need to pass to a birational model.  This leads to a robust notion of a Zariski decomposition for curve classes (see \cite{lx15} for more details).
\end{exmple}


The following properties are elementary.

\begin{lem} \label{volbasicprops}
Let $X$ be an integral projective variety of dimension $n$ and suppose $\alpha \in N_{k}(X)$ for $0 \leq k < n$.
\begin{enumerate}
\item For any positive constant $c$ we have $\widehat{\vol}(c\alpha) = c^{n/n-k} \widehat{\vol}(\alpha)$.
\item If $\alpha$ is big then $\widehat{\vol}(\alpha) > 0$.
\item If $\alpha$ is pseudo-effective, then for any class $\beta \in N_{k}(X)$ we have $\widehat{\vol}(\alpha + \beta) \geq \widehat{\vol}(\beta)$.
\end{enumerate}
\end{lem}




The main point is to understand the behavior of $\widehat{\vol}$ as we approach the boundary.  This is controlled by the following theorem.

\begin{thrm} \label{volmobineq}
Let $X$ be an integral projective variety of dimension $n$.  For any $\alpha \in \Eff_{k}(X)$ we have
\begin{equation*}
\widehat{\vol}(\alpha) \leq \mob(\alpha).
\end{equation*}
\end{thrm}

\begin{proof}
Combine the computation $\mob(H^{n-k}) \geq \vol(H)$ with the fact that mobility can only increase upon pushforward.
\end{proof}

\begin{cor}
Let $X$ be an integral projective variety.  Then $\widehat{\vol}$ is a continuous function on $N_{k}(X)$ for any $0 \leq k < n$.
\end{cor}

\begin{proof}
Lemma \ref{volbasicprops} verifies the hypotheses of Lemma \ref{easyconelem}, showing that $\widehat{\vol}$ is continuous on the interior of the big cone.  Theorem \ref{mobcontinuous} and Theorem \ref{volmobineq} show that $\widehat{\vol}$ must limit to $0$ as we approach the boundary of the pseudo-effective cone.
\end{proof}

\section{Weighted mobility} \label{weightedmobilitysection}

One approach for calculating the mobility is to study the geometry of the blow-up of $X$ through general points, but unfortunately this does not seem very effective.  One can obtain a closer relationship using a ``weighted'' mobility count giving singularities a higher contribution.  
This idea was suggested to me by R.~Lazarsfeld and this section is based off his suggestion.



\subsection{Seshadri constants}
We start with a few reminders about Seshadri constants at general points; see \cite{lazarsfeld04} or \cite{7authors} for a more thorough introduction to the area.

\begin{defn}
Let $X$ be an integral projective variety of dimension $n$ and let $A$ be an ample Cartier divisor on $X$.  Fix distinct closed reduced points $\{x_{i} \}_{i = 1}^{b}$ in the smooth locus of $X$.  Set $\phi: Y \to X$ to be the blow-up of the $x_{i}$ and let $E$ denote the sum of all the exceptional divisors.  The Seshadri constant of $A$ along the $\{ x_{i} \}$ is
\begin{equation*}
\varepsilon( \{ x_{i} \}, A) := \max  \left\{ r \in \mathbb{R}_{\geq 0}  \left| \phi^{*}A - r E \textrm{ is nef} \right. \right\}.
\end{equation*}
\end{defn}

We define $\varepsilon_{b}(A)$ to be the supremum over all sets of $b$ distinct closed points $\{ x_{i} \}_{i=1}^{b}$ in the smooth locus of $X$ of $\varepsilon(\{ x_{i} \}, A)$.  When we are working over an uncountable field there is actually a set of points achieving this supremum, but we will not need this fact.

It is an important but difficult problem to precisely establish the value of $\varepsilon_{b}(A)$.  The following estimate gives a good asymptotic bound on $\varepsilon_{b}(A)$.  It is certainly well-known to experts and the first variant for higher dimension varieties seems to have appeared in \cite{angelini97}.

\begin{prop} \label{seshadricalculationveryample}
Let $X$ be an integral projective variety of dimension $n$ over an uncountable algebraically closed field and let $A$ be a very ample divisor on $X$.  Suppose that $b \leq t^{n}A^{n}$ for a positive integer $t$.  Fix general points $\{ x_{i} \}_{i=1}^{b}$ on $X$.  Then
\begin{equation*}
\frac{1}{t} \leq \varepsilon(\{ x_{i} \},A) \leq \frac{(A^{n})^{1/n}}{b^{1/n}}.
\end{equation*}
In particular $\frac{1}{t} \leq \varepsilon_{b}(A) \leq \frac{(A^{n})^{1/n}}{b^{1/n}}$.
\end{prop}

\subsection{Loci of points with multiplicity}

If $Z$ is an integral projective variety and $z \in Z$ is a reduced closed point, we denote by $\mult(Z,z)$ the multiplicity of $Z$ at $z$ (that is, the multiplicity of the maximal ideal $\mathfrak{m}_{Z,z}$ in the local ring $\mathcal{O}_{Z,z}$).  

More generally, we define the multiplicity of a $k$-cycle $V = \sum_{i} a_{i}V_{i}$ at a reduced closed point $v \in \Supp(V)$ via the expression $\sum a_{i} \mult(V_{i},v)$.  Note that this definition is compatible with the blow-up definition of multiplicity. 

\begin{lem}
Suppose that $W$ is an integral variety.  Let $p: U \to W$ be a family of effective $k$-cycles on $X$ and let $U_{[w]}$ denote the cycle-theoretic fiber above $w$.  The function on reduced closed points
\begin{equation*}
u \mapsto \mult(U_{[f(u)]},u)
\end{equation*}
is upper semi-continuous.
\end{lem}

This lemma is proved by blowing up the reduced diagonal in $U \times U$ and using the fact that multiplicities of fibers in the exceptional locus can only jump up in closed subsets.



\begin{defn}
Let $X$ be an integral projective variety and let $p: U \to W$ denote a family of effective $k$-cycles on $X$.  Let $U_{\mu} \subset U$ denote the closed subset consisting of reduced closed points with multiplicity in the corresponding cycle-theoretic fiber at least $\mu$.  If $p_{\mu}: U_{\mu} \to W$ denotes the restriction of $p$, we define
\begin{equation*}
\mc(p;\mu) := \mc(p_{\mu}).
\end{equation*}
If $\alpha \in N_{k}(X)_{\mathbb{Z}}$, we define
\begin{equation*}
\mc(\alpha;\mu) := \sup_{p} \mc(p;\mu)
\end{equation*}
as we vary $p$ over all families of effective $k$-cycles representing $\alpha$.  (If there are no such families, we set $\mc(\alpha;\mu)=0$.)
\end{defn}

We note the following easy properties.

\begin{lem} \label{rescalingweightcount}
Let $X$ be an integral projective variety of dimension $n$ and let $\alpha \in N_{k}(X)_{\mathbb{Z}}$.  Then for any positive integer $\mu$:
\begin{enumerate}
\item If $\beta$ is represented by an effective cycle then $\mc(\alpha + \beta;\mu) \geq \mc(\alpha;\mu)$.
\item If both $\alpha$ and $\beta$ are represented by effective cycles then $\mc(\alpha + \beta;\mu) \geq \mc(\alpha;\mu) + \mc(\beta;\mu)$.
\item For any $r \in \mathbb{Z}_{>0}$ we have $\mc(\alpha;\mu) \leq \mc(r\alpha;r\mu)$.
\end{enumerate}
\end{lem}

\begin{proof}
The first two statements are obvious using the family sum construction since the multiplicity of a cycle at a point can only increase upon the addition of an effective cycle.  For the third, if a family $p$ represents $\alpha$, then by rescaling the coefficients by $r$ we obtain a family representing $r\alpha$, and it is clear the multiplicities go up by a factor of $r$.
\end{proof}

\subsection{Weighted mobility count}

\begin{defn} \label{wmcdef}
Let $X$ be an integral projective variety of dimension $n$ and let $\alpha \in N_{k}(X)_{\mathbb{Q}}$.  Define the weighted mobility count of $\alpha$ to be
\begin{equation*}
\wmc(\alpha) = \sup_{\mu} \mc(\mu \alpha; \mu).
\end{equation*}
where we vary $\mu$ over all positive integers such that $\mu \alpha \in N_{k}(X)_{\mathbb{Z}}$.
\end{defn}

We next show that $\wmc(\alpha)$ is always finite.  This implies that we can always choose a multiplicity $\mu$ maximizing the weighted mobility count, and in particular, the weighted mobility count can be computed by a double supremum as in the introduction (by combining Lemma \ref{wmcgrowth} with Lemma \ref{rescalingweightcount}).  In fact, we prove a weak upper bound which also gives the correct asymptotic rate of growth.

\begin{lem} \label{wmcgrowth}
Let $X$ be an integral projective variety of dimension $n$ and let $\alpha \in N_{k}(X)_{\mathbb{Q}}$.  Fix a very ample Cartier divisor $A$ on $X$.  Then
\begin{equation*}
\wmc(\alpha) \leq \sup \left\{ A^{n}, \left( \frac{2}{(A^{n})^{1/n}} \right)^{nk/n-k} (A^{k} \cdot \alpha)^{n/n-k} \right\}.
\end{equation*}
\end{lem}

\begin{proof}
Fix $b$ general points $\{ x_{i} \}_{i=1}^{b}$ on $X$ and let $\phi: Y \to X$ be the blow-up with total exceptional divisor $E$.  Suppose that $Z$ is an effective cycle of class $\mu \alpha$ with multiplicity $\geq \mu$ at each of the $b$ points.  Since $\phi^{*}A - \varepsilon(\{ x_{i} \},A)E$ is nef, the strict transform $Z'$ of $Z$ satisfies
\begin{equation*}
0 \leq (\phi^{*}A - \varepsilon(\{ x_{i} \},A)E)^{k} \cdot Z' = A^{k} \cdot \mu \alpha - b \mu \varepsilon(\{x_{i} \},A)^{k}.
\end{equation*}
Proposition \ref{seshadricalculationveryample} shows that either $b \leq A^{n}$ (when $t=1)$ or $\varepsilon(\{ x_{i} \},A) \geq \frac{1}{2}(A^{n})^{1/n}b^{-1/n}$ (when $t>1$ so that $\frac{t-1}{t} \geq \frac{1}{2}$).  In the second case, for any family $p_{\mu}$ representing $\mu \alpha$, we obtain
\begin{equation*}
\mc(p_{\mu};\mu) \leq \left( \frac{2}{(A^{n})^{1/n}} \right)^{nk/n-k} (A^{k} \cdot \alpha)^{n/n-k}.
\end{equation*}
Since this expression is independent of $\mu$ we obtain the proof.
\end{proof}

\begin{lem} \label{wmcsubadditivity}
Let $X$ be an integral projective variety of dimension $n$ and let $\alpha, \beta \in \Eff_{k}(X)_{\mathbb{Q}}$.  Suppose that some multiple of $\alpha - \beta$ is represented by an effective cycle.  Then $\wmc(\alpha) \geq \wmc(\beta)$.
\end{lem}

\begin{proof}
Choose a sufficiently divisible integer $\mu$ so that $\mu(\alpha-\beta)$ is represented by an effective cycle, $\wmc(\alpha) = \mc(\mu \alpha;\mu)$, and $\wmc(\beta) = \mc(\mu \beta;\mu)$.  Then $\mc(\mu \alpha;\mu) \geq \mc(\mu \beta;\mu)$ by Lemma \ref{rescalingweightcount}.
\end{proof}

\subsection{Weighted mobility}

Lemma \ref{wmcgrowth} indicates that we should take the following asymptotic definition.

\begin{defn}
Let $X$ be an integral projective variety of dimension $n$ and let $\alpha \in  N_{k}(X)_{\mathbb{Q}}$.  The weighted mobility of $\alpha$ is:
\begin{equation*}
\wmob(\alpha) := \limsup_{m \to \infty} \frac{\wmc(m\alpha)}{m^{n/n-k}}.
\end{equation*}
\end{defn}

Note that there is no longer a factor of $n!$.  Lemma \ref{wmcgrowth} shows that the weighted mobility is always finite.
The weighted mobility satisfies the same properties as $\mob$ with the same proofs; we make brief verifications when necessary.

\begin{lem} \label{rescalingwmob}
Let $X$ be an integral projective variety and let $\alpha \in N_{k}(X)_{\mathbb{Q}}$.  Fix a positive integer $a$.  Then $\wmob(a\alpha) = a^{\frac{n}{n-k}}\wmob(\alpha)$.
\end{lem}


\begin{lem} \label{wmobadditive}
Let $X$ be an integral projective variety.  Suppose that $\alpha, \beta \in N_{k}(X)_{\mathbb{Q}}$ are classes such that some positive multiple of each is represented by an effective cycle.  Then $\wmob(\alpha + \beta) \geq \wmob(\alpha) + \wmob(\beta)$.
\end{lem}


It is not hard to show that a complete intersection of ample divisors always has positive $\wmob$; a precise computation is done in Example \ref{wmobci}.  By Lemma \ref{wmobadditive} we obtain:

\begin{cor}
Let $X$ be an integral projective variety and let $\alpha \in N_{k}(X)_{\mathbb{Q}}$ be a big class.  Then $\wmob(\alpha) > 0$.
\end{cor}

Then Lemma \ref{easyconelem} shows:

\begin{thrm}
Let $X$ be an integral projective variety.  The function $\wmob: N_{k}(X)_{\mathbb{Q}} \to \mathbb{R}_{\geq 0}$ is locally uniformly continuous on the interior of $\Eff_{k}(X)_{\mathbb{Q}}$.
\end{thrm}

Note that the multiplicity of a cycle at a general point can only increase upon taking a birational pushforward.  Since mobility counts also can only increase upon pushforward, we have

\begin{prop} \label{wmobincreasepushforward}
Let $\pi: X \to Y$ be a surjective birational morphism of integral projective varieties.  For any $\alpha \in N_{k}(X)_{\mathbb{Q}}$ we have $\wmob(\pi_{*}\alpha) \geq \wmob(\alpha)$.
\end{prop}

\subsection{Continuity of weighted mobility}

We analyze the continuity of the weighted mobility function using similar arguments as for the mobility.  Again, the base case is in codimension $1$.

\begin{lem} \label{weighteddivisorrestriction}
Let $X$ be an integral projective variety of dimension $n$.  Let $\alpha \in N_{n-1}(X)_{\mathbb{Q}}$ and suppose that $A$ is a very ample divisor and $s$ is a positive integer such that $\alpha \cdot A^{n-1} < sA^{n}$.  Fix a positive integer $\mu$ such that $\mu \alpha \in N_{n-1}(X)_{\mathbb{Z}}$.  Fix a general element $H \in |sA|$.  For $0 \leq i \leq \mu-1$ there are positive integers $k_{i}$ and collections of families of effective $(n-2)$-cycles $\{ r_{i,j} \}_{j=1}^{k_{i}}$ on $H$  such that $[r_{i,j}] \cdot A|_{H}^{n-2} < (\mu-i) sA|_{H}^{n-1}$ and
\begin{equation*}
\mc(\mu \alpha; \mu) \leq \sup_{i,j} \mc_{H}(r_{i,j};\mu-i).
\end{equation*}
\end{lem}

The $r_{i,j}$ are constructed by taking a family representing $\mu \alpha$, considering the subfamilies in which $H$ occurs with multiplicity $i$, and then for each removing $H$ (with the appropriate multiplicity) and restricting to $H$.

\begin{proof}
For a general $H \in |sA|$ we have that $H$ is integral.  Let $p: U \to W$ be a family of effective $(n-1)$-cycles representing $\mu \alpha$ realizing the weighted mobility count.  By Lemma \ref{goodfamilymodification} we may suppose that $U \to X$ is projective and $W$ is normal projective.  Set $p_{\mu}: U_{\mu} \to W$ to be the closed subset of points whose multiplicity in the corresponding cycle-theoretic fiber is at least $\mu$.  Let $\widehat{U}_{\mu}$ denote the union of the components of $U_{\mu}$ which are not contained in the singular locus of $H$ and denote by $\widehat{p}_{\mu}$ the restriction of $p_{\mu}$; note that removing such components will not affect mobility counts.  Set $b = \mc(p;\mu)$ so that
\begin{equation*}
\widehat{U}_{\mu}^{ \times_{W} b} \to X^{\times b}
\end{equation*}
is surjective.  

Stratify $W$ into locally closed subsets $W_{i}$ such that $H$ has multiplicity exactly $i$ in every fiber of $p$ over a closed point of $W_{i}$.  Note that since $\mu \alpha - \mu H$ is not pseudo-effective, $W_{i}$ is empty for $i \geq \mu$.  Suppose that $W_{i}$ has $k_{i}$ irreducible components enumerated as $W_{i,j}$ and consider the restriction of the family $p$ to $W_{i,j}$.  This restricted family has one component corresponding to the constant divisor $iH$; after removing this divisor, we obtain a family $p_{i,j}$ of divisors on $X$ not containing $H$ in their support.  Note that $[p_{i,j}] + is[A] = \mu \alpha$.

Replace $p_{i,j}$ by a projective normalized closure as in Lemma \ref{goodfamilymodification}.  For each family, let $p_{i,j,\mu-i}$ denote the closed locus of points whose multiplicities in the fibers is at least $\mu-i$.  We claim that every fiber of $\widehat{p}_{\mu}$ over $W$ is set theoretically contained in a fiber of some $p_{i,j,\mu-i}$.  Indeed, since $\mu \alpha - \mu H$ is not pseudo-effective, we see that any point of multiplicity $\mu$ in a fiber of our original family $p$ which is not contained in a singular point of $H$ must have a contribution from components aside from $H$ of multiplicity at least $\mu-i$.  Thus, arguing as in Lemma \ref{fibercontainedmc} we obtain a finite collection of closed subvarieties $U_{i,j,\mu - i}$ of $W_{i,j} \times X$ such that
\begin{equation*}
\bigcup_{i,j} U_{i,j,\mu-i}^{ \times_{\cup W_{i,j}} b} \to X^{\times b}
\end{equation*}
is surjective.  (In other words, away from the singular locus of $H$ any point which has fiberwise multiplicity $\geq \mu$ in our original family must coincide with a point which has multiplicity $\geq \mu-i$ in one of our new families.)  The base change of a surjective map is again surjective; in this way we obtain closed subsets $U^{H}_{i,j,\mu-i}$ such that the $b$th relative product over the base of the family maps surjectively onto $H^{\times b}$.

Recall that the support of divisors in the family $p_{i,j}$ never contains $H$.  By intersecting the family $p_{i,j}$ with the divisor $H$ we obtain a family $r_{i,j}: Q_{i,j} \to T_{i,j}$ of $(n-2)$-cycles on $H$ satisfying
\begin{align*}
[r_{i,j}] \cdot A|_{H}^{n-2} & = [p_{i,j}] \cdot sA^{n-1} = s \mu \alpha \cdot A^{n-1} - i s^{2} A^{n} \\
& < (\mu - i) s^{2}A^{n} = (\mu-i) sA|_{H}^{n-1}.
\end{align*}
Replace $r_{i,j}$ by a projective normalized closure as in Lemma \ref{goodfamilymodification}, and note that every intersection of a member of $p_{i,j}$ with $H$ is contained in the fiber of some $r_{i,j}$.  Since the multiplicity of a cycle-theoretic fiber along a point in $H$ can only increase upon intersection with $H$, we see that $Q_{i,j,\mu-i}$ set-theoretically contains the base change $U^{H}_{i,j,\mu-i}$.  Again applying Lemma \ref{fibercontainedmc} we obtain the desired statement.
\end{proof}




\begin{prop} \label{mcbasicvarestimate}
Let $X$ be an integral projective variety of dimension $n$ and let $\alpha \in N_{k}(X)_{\mathbb{Q}}$.  
\begin{enumerate}
\item Suppose that $A$ is a very ample Cartier divisor on $X$ and $s$ is a positive integer such that $\alpha \cdot A^{n-1} < sA^{n}$.  Then
\begin{equation*}
\wmc(\alpha) < s^{n} A^{n}.
\end{equation*}
\item Suppose $n \geq 2$.  Let $A$ and $H$ be very ample divisors and let $s$ be a positive integer such that $\alpha - [H]$ is not pseudo-effective and $\alpha \cdot A^{n-2} \cdot H < s A^{n-1} \cdot H$.  Then
\begin{equation*}
\wmc(\alpha) < s^{n-1} A^{n-1} \cdot H.
\end{equation*}
\end{enumerate}
\end{prop}

\begin{proof}
(1) The proof is by induction on the dimension of $X$.  If $X$ is a curve, then for any class $\beta \in N_{0}(X)_{\mathbb{Z}}$
\begin{equation*}
\mc(\beta;\mu) = \left\lfloor \frac{\deg(\beta)}{\mu} \right\rfloor \leq \frac{\deg(\beta)}{\mu}.
\end{equation*}
Thus
\begin{equation*}
\wmc(\alpha) = \sup_{\mu \textrm{ sufficiently divisible}} \mc(\mu \alpha;\mu) < s\deg(A).
\end{equation*}
In general, applying Lemma \ref{weighteddivisorrestriction} to $sA$ (and keeping the notation there) we find for any sufficiently divisible $\mu$
\begin{align*}
\mc(\mu \alpha; \mu) & \leq \sup_{i,j} \mc_{H}(r_{i,j};\mu-i) \\
& \leq \sup_{i,j} \wmc_{H}([r_{i,j}]/\mu-i) \\
& < s^{n-1}A|_{H}^{n-1} = s^{n}A^{n}
\end{align*}
where the final inequality follows from induction.

(2) This follows by a similar argument by applying Lemma \ref{weighteddivisorrestriction} to $H$, then applying (1) to $\alpha$ and $A$ restricted to $H$.
\end{proof}

Define the constants $\epsilon_{n,k}$, $\tau_{n,k}$ as in Theorem \ref{mobprecisebound}.

\begin{thrm} \label{wmobprecisebound}
Let $X$ be an integral projective variety and let $\alpha \in N_{k}(X)_{\mathbb{Q}}$.  Let $A$ be a very ample divisor and let $s$ be a positive integer such that $2^{n}\alpha \cdot A^{k} < sA^{n}$.  Then
\begin{enumerate}
\item
\begin{equation*}
\wmc(\alpha) <  2^{kn+3n} s^{\frac{n}{n-k}}A^{n} 
\end{equation*}
\item Suppose furthermore that $2^{n}\alpha - [A]^{n-k}$ is not pseudo-effective.  Then
\begin{equation*}
\wmc(\alpha) < 2^{kn+3n} s^{\frac{n}{n-k} - \epsilon_{n,k}} A^{n}. 
\end{equation*}
\item Suppose that $t$ is a positive integer such that $t \leq s$ and $2^{n}\alpha - t[A]^{n-k}$ is not pseudo-effective.  Then
\begin{equation*}
\wmc(\alpha) < 2^{kn+3n}  s^{\frac{n}{n-k} - \tau_{n,k}} t^{\tau_{n,k}} A^{n}. 
\end{equation*}
\end{enumerate}
\end{thrm}

The proof is essentially the same as the proof of Theorem \ref{mobprecisebound}.  The key point is to understand how the weighted mobility count changes upon specializing our cycles into the hypersurface $H$; we will only highlight the necessary changes.

\begin{proof}
Choose a multiplicity $\mu$ and a family $p: U \to W$ representing $\mu \alpha$ such that $\wmc(\alpha) = \mc(p;\mu)$.  Retain the constructions and notation of the proof of Theorem \ref{mobprecisebound} for the family $p$ and the divisor $H$.  Thus we have a family of $(k-1)$-cycles $q: S \to T$ and families of $k$-cycles $p_{i}: V_{i} \to W_{i}$ which between them parametrize all the components of intersections of members of $p$ with $H$.

Consider a fixed cycle $\sum_{i=1}^{r} a_{i} V_{i}$ in the family $p$.  We may suppose that the first $r'$ of these cycles are the components contained in $H$.  Then for any point $x \in \Supp(V_{i}) \cap H$,
\begin{align*}
\mult(V,x) & = \sum_{i=1}^{r} a_{i} \mult(V_{i},x) \\
& \leq \sum_{i=1}^{r'} a_{i} \mult(V_{i},x) + \sum_{i=r'}^{r} a_{i} \mult(V_{i} \cdot H, x)
\end{align*}
as multiplicities can only increase upon intersection with a hyperplane.  At least one of the terms on the right is $\geq \frac{1}{2} \mult(V,x)$.  Thus, we see that every fiber of $p_{\mu}: U_{\mu} \to W$, where $U_{\mu}$ denotes the locus of points which have fiberwise multiplicity $\geq \mu$, is contained set theoretically in the union over all $j$ of the loci in $S$ of points which have fiberwise multiplicity at least $\mu/2$ and the loci in $V_{i}$ of points which have fiberwise multiplicity at least $\mu/2$.  Arguing in families, we have
\begin{equation*}
\mc(\mu \alpha;\mu) \leq  \mc(\mu\alpha \cdot H;  \mu/2 ) + \sup_{i} \mc_{H}([p_{i}];  \mu/2 ).
\end{equation*}
It is clear that $\mc(\mu \alpha \cdot H;  \mu/2 ) \leq \wmc_{H}(2\alpha \cdot H)$ and $\sup_{i} \mc_{H}([p_{i}]; \mu/2 ) \leq \sup_{i_{*}\beta \preceq \alpha} \wmc_{H}(2\beta)$.

At this point the proof of (1), (2), (3) proceeds exactly as in Theorem \ref{mobprecisebound}, but with some additional factors of $2$:
\begin{itemize}
\item In the proof of (1), we must account for the halving of the multiplicity (or equivalently, the potential doubling of the integer $s$) at each step inductively.   This is accomplished by the factor of $2^{n}$ in the inequality for $\alpha \cdot A^{k}$; the constant $s$ is then preserved by the inductive step.

\item In the proof of (2) we need to ensure inductively that while adding the coefficient of $2$ to the families $p_{i}$ the hypothesis of (2) still holds for the new families in the new ambient variety $H$.  Again, the easiest way to do this is simply to ensure that $2^{n}\alpha - [A]$ is not pseudo-effective.

\item The presence of the factor $2^{n}$ still exactly preserves the inductive structure of the argument for (3).
\end{itemize}
\end{proof}

Arguing just as in the proof of Theorem \ref{mobcontinuous}, we find:

\begin{thrm} \label{wmobcontinuous}
Let $X$ be an integral projective variety.  Then the weighted mobility function $\wmob: N_{k}(X)_{\mathbb{Q}} \to \mathbb{R}$ can be extended to a continuous function on $N_{k}(X)$.
\end{thrm}

\subsection{Computations of weighted mobility}

We now compute the weighted mobility in two special examples: for complete intersections of ample divisors and for big divisors on a smooth variety.  For ease of notation we work over an uncountable algebraically closed field (although the computation would work equally well over any algebraically closed field using a slight perturbation of $\varepsilon_{b}$).

\begin{exmple} \label{wmobci}
Let $X$ be an integral projective variety of dimension $n$ over an uncountable algebraically closed field and let $H$ be a big and nef $\mathbb{R}$-Cartier divisor.  Set $\alpha = [H^{n-k}]$.  We show that
\begin{equation*}
\wmob(\alpha) = \vol(H).
\end{equation*}
Using continuity and homogeneity it suffices to consider the case when $H$ is very ample.

We first show the inequality $\leq$.  Suppose that $Z$ is an effective $\mathbb{Z}$-cycle of class $m \mu \alpha$ which goes through $b$ general points of $X$ with multiplicity $\geq \mu$ at each.  Let $\phi: Y \to X$ be the blow-up of $b$ very general points and let $E$ denote the sum of the exceptional divisors.  Then the strict transform $Z'$ satisfies
\begin{equation*}
0 \leq Z' \cdot (\phi^{*}H - \varepsilon_{b}(H) E)^{k} \leq m \mu \vol(H) - b \mu \varepsilon_{b}(H)^{k}.
\end{equation*}
Choose a positive integer $t$ such that $(t-1)^{n}\vol(H) < b \leq t^{n}\vol(H)$.  
Proposition \ref{seshadricalculationveryample} shows
\begin{equation*}
\varepsilon_{b}(H) \geq \frac{1}{t}.
\end{equation*}
Combining the previous equations, we see that
\begin{equation*}
\frac{b}{t^{k}} \leq m \vol(H). 
\end{equation*}
If $t>1$ then we have the relationship $\frac{1}{(t-1)^{k}} > (\frac{\vol(H)}{b})^{k/n}$, yielding
\begin{equation*}
\left( \frac{t-1}{t} \right)^{kn/n-k} b \leq \vol(H) m^{n/n-k}
\end{equation*}
while if $t=1$ then $b \leq \vol(H)$.  For sufficiently large $b$, the left hand side approaches $b$.  More precisely, for any $\delta > 0$ there is a constant $b_{0}$ such that $(1-\delta)\wmc(m\alpha) \leq \vol(H) m^{n/n-k}$ as soon as $\wmc(m\alpha)$ is at least $b_{0}$.  Taking a limit, we find $\wmob(\alpha) \leq \vol(H)$.


To show the other inequality $\geq$, we need to construct complete intersection families on $X$.  Fix a positive integer $t$ and set $b = t^{n}\vol(H)$.  Let $\phi: Y \to X$ be the blow-up of $b$ very general points on $X$ and let $E$ be the sum of the exceptional divisors.  By Proposition \ref{seshadricalculationveryample} we see that $\varepsilon_{b}(H) = 1/t$.  Choose a sequence of rational numbers $\tau_{i}$ which limits to $\varepsilon_{b}(H)$ from beneath and choose integers $c_{i}$ such that $c_{i}(\phi^{*}H - \tau_{i}E)$ is very ample.  Take a complete intersection of elements of this very ample linear system and pushforward under $\phi$.  The result is an effective $\mathbb{Z}$-cycle of class $c_{i}^{n-k}\alpha$ which has multiplicity $\geq c_{i}^{n-k}\tau_{i}^{n-k}$ at each of the $b$ points.

Set $m = t^{n-k}$ and $\mu = (c_{i} \tau_{i})^{n-k}$.  If we fix $\tau_{i}$ and let $c_{i}$ vary over all sufficiently divisible integers, we find infinitely many values of $\mu$ for which there is a cycle of class $\frac{1}{\tau_{i}^{n-k}t^{n-k}} m \mu \alpha$ going through $m^{n/n-k}\vol(H)$ points with multiplicity $\geq \mu$ at each.  Note that $\tau_{i}^{n-k} t^{n-k}$ can be made arbitrarily close to $1$ for $i$ sufficiently large.  In other words, there is a sequence of positive rational numbers $\epsilon_{i}$ converging to $0$ such that for each $\epsilon_{i}$ there are (infinitely many) values of $\mu$ satisfying
\begin{equation*}
\mc(m \mu (1+\epsilon_{i}) \alpha; \mu) \geq m^{n/n-k}\vol(H).
\end{equation*}
This easily yields $\wmob(\alpha) \geq \vol(H)$.
\end{exmple}

\begin{exmple} \label{wmobdivisor}
Suppose that $X$ is a smooth projective variety of dimension $n$ over an uncountable algebraically closed field and that $L$ is a big $\mathbb{R}$-Cartier divisor on $X$.  Then
\begin{equation*}
\wmob([L]) = \vol(L).
\end{equation*}

By continuity on the big cone it suffices to consider the case when $L$ is $\mathbb{Q}$-Cartier.  Let $\phi: Y \to X$ and the ample $\mathbb{Q}$-Cartier divisor $A$ on $Y$ be an $\epsilon$-Fujita approximation for $X$ (constructed by \cite{takagi07} in arbitrary characteristic).  Example \ref{wmobci} shows that $\wmob([A]) = \vol(A)$, and pushing forward and applying Lemma \ref{wmobadditive} and Proposition \ref{wmobincreasepushforward} we see that $\wmob([L]) \geq \vol(L) - \epsilon$ for any positive $\epsilon > 0$.

The converse follows using $\epsilon$-Fujita approximations for $L$ and an argument similar to Example \ref{wmobci}.  
\end{exmple}

\bibliographystyle{amsalpha}
\bibliography{mobility}

\end{document}